\documentclass[11pt]{article}
\usepackage{amsmath,amssymb,a4wide,amsthm,color,graphicx,verbatim,algorithmic,algorithm,caption,enumerate,multirow}
\newtheorem{theorem}{Theorem}[section]

\newtheorem{remark}[theorem]{Remark}

\newtheorem{proposition}[theorem]{Proposition}
\newtheorem{ex}[theorem]{Example}
\newenvironment{example}{\begin{ex} \em }{\em \end{ex}}
\newtheorem{definition}[theorem]{Definition}
\newtheorem{assumption}[theorem]{Assumption}

\newcommand{\norm}  [1]{\ensuremath{\left  \|       #1  \right \|       }}

\newcommand{\cb}    [1]{\ensuremath{\left  \{      #1  \right \}       }}

\newcommand{\abs  } [1]{\ensuremath{\left  |       #1  \right |        }}

\newcommand{\st} {\ensuremath{|\;}}

\newcommand{\epi} {{\rm epi \,}}

\newcommand{\cl}  {{\rm cl  \,}}
\newcommand{\bd}  {{\rm bd \,}}

\newcommand{\Int} {{\rm int \,}}
\newcommand{\conv}  {{\rm conv \,}}
\newcommand{\cone}{{\rm cone\,}}

\newcommand{\E}{\mathcal{E}}
\newcommand{\G}{\mathcal{G}}
\newcommand{\R}{\mathbb{R}}
\newcommand{\N}{\mathbb{N}}
\newcommand{\smz}{\setminus\{0\}}

\renewcommand{\E}{\mathbb E}
\renewcommand{\P}{\mathbb P}
\newcommand{\co}{{\rm co \,}}

\newcommand{\wMinc}{{\rm wMin}_C\,}
\newcommand{\nwMinc}{{\rm (w)Min}_C\,}
\newcommand{\MaxK}{{\rm Max}_K\,}

\author{Andreas L\"{o}hne  \thanks{Martin-Luther-Universit{\"a}t Halle-Wittenberg, NWF II, Departement of Mathematics, 06099 Halle (Saale), andreas.loehne@mathematik.uni-halle.de} 
\and Birgit Rudloff \thanks{Princeton University, Department of Operations Research and Financial Engineering; and Bendheim Center for Finance, Princeton, NJ 08544, USA, brudloff@princeton.edu. Research supported by NSF award DMS-1007938.} 
\and Firdevs Ulus  \thanks{Princeton University, Department of Operations Research and Financial Engineering, Princeton, NJ 08544, USA, fulus@princeton.edu}
}

\title{Primal and Dual Approximation Algorithms for Convex Vector Optimization Problems}

\date{\today}

\begin{document}
\maketitle

\begin{abstract} \noindent
Two approximation algorithms for solving convex vector optimization problems (CVOPs) are provided. Both algorithms solve the CVOP and its geometric dual problem simultaneously. The first algorithm is an extension of Benson's outer approximation algorithm, and the second one is a dual variant of it. Both algorithms provide an inner as well as an outer approximation of the (upper and lower) images. Only one scalar convex program has to be solved in each iteration. We allow objective and constraint functions that are not necessarily differentiable, allow solid pointed polyhedral ordering cones, and relate the approximations to an appropriate $\epsilon$-solution concept. 
Numerical examples are provided. 

\medskip

\noindent
{\bf Keywords:} Vector optimization, multiple objective optimization, convex programming,
duality, algorithms, outer approximation.

\medskip

\noindent
{\bf MSC 2010 Classification:} 90C29, 90C25, 90-08, 91G99

\end{abstract}

\section{Introduction}

A variety of methods have been developed in the last decades to solve or approximately solve vector optimization problems. As in scalar optimization, only special problem classes are tractable. One of the most studied classes consists of linear vector optimization problems (LVOPs). There are many solution methods for LVOPs in the literature, see e.g. the survey paper by Ehrgott and Wiecek \cite{survey_ehrgott} and the references therein. The multi-objective simplex method, for instance, evaluates the set of all \emph{efficient solutions} in the variable space (or decision space). Interactive methods, for instance, compute a sequence of efficient solutions depending on the decision maker's preferences. Benson \cite{benson} proposed an outer approximation algorithm in order to generate the set of all \emph{efficient values} in the objective space. He motivates this method by observing that typically the dimension of the objective space is much smaller than the dimension of the variable space, decision makers tend to choose a solution based on objective values rather than variable values, and often many efficient solutions are mapped to a single efficient point in the objective space. A solution concept for LVOPs which takes into account these ideas has been introduced in \cite{lohne}. Several variants of Benson's algorithm for LVOPs have been developed, see e.g. \cite{shaoLVOP, dualalgorithm, lohne, ehr_dual}, having in common that at least two LPs need to be solved in each iteration. Independently in \cite{lvop} and \cite{csirmaz}, an improved variant for LVOPs has been proposed where only one LP has to be solved in each iteration. 

Convex vector optimization problems (CVOPs) are more difficult to solve than LVOPs. There are methods which deal with CVOPs, and specific subclasses of it. We refer the reader to the survey paper by Ruzika and Wiecek \cite{survey_ruzika} for a classification of approximation methods for a CVOP. Recently, Ehrgott, Shao, and Sch\"obel \cite{ehrgott} developed a Benson type algorithm for bounded CVOPs, motivated by the same arguments as given above for LVOPs. They extended Benson's algorithm to approximate the set of all efficient values in the objective space from LVOPs to CVOPs. In this paper, we generalize and simplify this approximation algorithm and we introduce a dual variant of it. To be more detailed, compared to \cite{ehrgott}, we 

\vspace{-0.3cm}
\begin{enumerate}[(i)]
	\setlength{\itemsep}{-2mm}
\item allow objective and constraint functions that are not necessarily differentiable (provided certain non-differentiable scalar problems can be solved), 
\item allow more general ordering cones,
\item use a different measure for the approximation error (as in \cite{lvop}), which allows the algorithms to be applicable to a larger class of problems,
\item obtain at no additional cost a finer approximation by including suitable points in the primal and dual inner approximations throughout the algorithm,
\item reduce the overall cost by simplifying the algorithm in the sense that only one convex optimization problem has to be solved in each iteration instead of two,
\item relate the approximation to an $\epsilon$-solution concept involving infimum attainment,
\item present additionally a dual algorithm that provides an alternative approximation/ $\epsilon$-solution.
\end{enumerate}

\vspace{-0.3cm}
This paper is structured as follows. Section~\ref{secPrelim} is dedicated to basic concepts and notation. In Section~\ref{secCVO}, the convex vector optimization problem, its geometric dual, and solution concepts are introduced. Furthermore, the geometric duality results for CVOPs are stated and explained. In Section~\ref{secalgo}, an extension of Benson's algorithm for CVOPs and a dual variant of the algorithm are provided. 
Numerical examples are given in Section~\ref{secex}. 

\section{Preliminaries}
\label{secPrelim}

For a set $A \subseteq \R^q$, we denote by $\Int A$, $\cl A$, $\bd A$, $\conv A$, $\cone A$, respectively the interior, closure, boundary, convex hull and the conic hull of $A$. A polyhedral convex set $A \subseteq \R^q$ can be defined as the intersection of finitely many half spaces, that is,
\vspace{-0.1cm}
\begin{align}\label{Hrep}
A = \bigcap_{i=1}^r \{ y \in \R^q : (w^i)^Ty \geq \gamma_i \}
\end{align}

\vspace{-0.1cm}\noindent
for some $r \in \N$, $w^1, \ldots, w^r \in \R^q \setminus \{0\}$, and $\gamma_1, \ldots, \gamma_r \in \R$. 
Every non-empty polyhedral convex set $A$ can also be written as
\vspace{-0.1cm}
\begin{align}\label{Vrep}
A = \conv\{x^1,\ldots, x^s\}+\cone \{k^1,\ldots,k^t\},
\end{align}

\vspace{-0.1cm}\noindent
where $s \in \N \setminus \{0\}, t \in \N $, each $x^i \in \R^q$ is a point, and each $k^j \in \R^q \setminus \{0\}$ is a \emph{direction} of $A$. Note that $k \in \R^q\setminus\{0\}$ is called a direction of $A$ if $A + \{\alpha k \in \R^q: \alpha > 0 \} \subseteq A$. The set of points $\{ x^1,\ldots, x^s \}$ together with the set of directions $\{k^1, \ldots , k^r \}$ are called the \emph{generators} of the polyhedral convex set $A$. Representation~\eqref{Vrep} of A is called the V-representation (or generator representation) whereas representation~\eqref{Hrep} of A by half-spaces is called H-representation (or inequality representation). A subset $F$ of a convex set $A$ is called an \emph{exposed face} of $A$ if there exists a supporting hyperplane $H$ to $A$, with $F = A \cap H$. 

A convex cone $C$ is said to be \emph{solid}, if it has a non-empty interior; \emph{pointed} if it does not contain any line; and \emph{non-trivial} if $\{0\} \subsetneq C \subsetneq \R^q$. A non-trivial convex pointed cone C defines a partial ordering $\leq_C$ on $\R^q$: $v \leq_C w$ if and only if $w - v \in C$. Let $C\subseteq \R^q$ be a non-trivial convex pointed cone and $X \subseteq \R^n$ a convex set. A function $\Gamma:X \rightarrow \R^q$ is said to be \emph{$C$-convex} if $\Gamma(\alpha x+(1-\alpha)y) \leq_C \alpha \Gamma(x)+(1-\alpha)\Gamma(y)$ holds for all $x,y \in X$, $\alpha \in [0,1]$, see e.g. \cite[Definition 6.1]{luc}. A point $y \in A$ is called \emph{$C$-minimal element} of A if $\left( \{y\} -  C \setminus \{0\}\right) \cap A = \emptyset$. If the cone $C$ is solid, then a point $y \in A$ is called \emph{weakly $C$-minimal element} if $ \left( \{y\} - \Int C \right) \cap A = \emptyset$. The set of all (weakly) $C$-minimal elements of $A$ is denoted by $\nwMinc(A)$. The set of (weakly) $C$-maximal elements is defined by ${\rm (w)Max}_C\,(A):={\rm (w)Min}_{-C}\,(A)$. The (positive) dual cone of $C$ is the set $C^+:=\cb{z \in \R^q\st \forall y \in C: z^T y \geq 0}$.

\section{Convex Vector Optimization}
\label{secCVO}

\subsection{Problem Setting and Solution Concepts}
\label{subsetP}

A convex vector optimization problem (CVOP) with polyhedral ordering cone $C$ is to
\vspace{-0.1cm}
\begin{align*} \label{(P)}
 \text{minimize~} \Gamma(x) \text{~with respect to~} \leq_C \text{~subject to~}  g(x) \leq_D 0, \tag{P}
\end{align*}

\vspace{-0.1cm}\noindent
where $C\subseteq\R^q$, and $D\subseteq\R^m$ are non-trivial pointed convex ordering cones with nonempty interior, $X\subseteq \R^n$ is a convex set, the vector-valued objective function $\Gamma: X \rightarrow \R^q$ is $C$-convex, and the constraint function $g: X \rightarrow \R^m$ is $D$-convex (see e.g. \cite{luc}). Note that the feasible set $\mathcal{X}:=\{x\in X: g(x)\leq_D 0 \} \subseteq X \subseteq\R^n$ of \eqref{(P)} is convex. Throughout we assume that \eqref{(P)}  is feasible, i.e., $\mathcal{X}\neq \emptyset$. The image of the feasible set is defined as $\Gamma(\mathcal{X}) = \{ \Gamma(x) \in \R^q : x \in \mathcal{X}\}$. The set 
\begin{equation}\label{upim1}
 \mathcal{P} := \cl(\Gamma(\mathcal{X})+C)	
\end{equation}
is called the \emph{upper image} of \eqref{(P)} (or {\em upper closed extended image of \eqref{(P)}}, see \cite{geometric}). Clearly, $\mathcal{P}$ is convex and closed.

\begin{definition}\label{bdd}
\eqref{(P)} is said to be \emph{bounded} if $\mathcal{P} \subseteq \{y\} + C$ for some $y \in \R^q$.
\end{definition}

The following definition describes a solution concept for complete-lattice-valued optimization problems which was introduced in \cite{freshlook}. It applies to the special case of vector optimization (to be understood in a set-valued framework). The solution concept consists of two components, minimality and infimum attainment. Here we consider the special case based on the complete lattice $\G(\R^q,C):=\{A \in \R^q: A=\cl\co(A+C)\}$ with respect to the ordering $\supseteq$, see e.g. \cite{lagrange}, which stays in the background in order to keep the notation simple.

\begin{definition}[\cite{lagrange,freshlook}] \label{soln}
A point $\bar{x}\in \mathcal{X}$ is said to be a \emph{(weak) minimizer} for \eqref{(P)} if $\Gamma(\bar{x})$ is a (weakly) $C$-minimal element of $\Gamma(\mathcal{X})$. 
A nonempty set $\mathcal{\bar{X}}\subseteq \mathcal{X}$ is called an \emph{infimizer of \eqref{(P)}} if 
$\cl \conv (\Gamma(\mathcal{\bar{X}})+C) = \mathcal{P}$.
An infimizer $\mathcal{\bar{X}}$ of \eqref{(P)} is called \emph{(weak) solution} to \eqref{(P)} if it consists of only (weak) minimizers.
\end{definition}

In \cite{lohne} this concept was adapted to linear vector optimization problems, where one is interested in solutions which consist of only finitely many minimizers. In the unbounded case, this requires to work with both points and directions in $\R^q$. For a (bounded) CVOP, the requirement that a solution consists of only finitely many minimizers is not adequate, since it is not possible in general to represent the upper image by finitely many points (and directions). However, a finite representation is possible in case of approximate solutions. To this end, we extend the idea of (finitely generated) $\epsilon$-solutions of LVOPs, which was introduced in Remark $4.10$ in \cite{lvop}, to the setting of bounded CVOPs. An infimizer has to be replaced by a {\em finite $\epsilon$-infimizer}. As we only consider bounded problems here, we do not need to deal with directions. It is remarkable that an $\epsilon$-solution provides both an inner and an outer polyhedral approximation of the upper image by finitely many minimizers.
Throughout this paper let $c \in \Int C$ be fixed.

\begin{definition}\label{weaksoln}
For a bounded problem \eqref{(P)}, a nonempty finite set $\mathcal{\bar{X}}\subseteq \mathcal{X}$ is called a \emph{finite $\epsilon$-infimizer of \eqref{(P)}} if 
\begin{equation} \label{eqweaksoln}
\conv \Gamma(\mathcal{\bar{X}})+C-\epsilon \{c\} \supseteq \mathcal{P}.
\end{equation}
A finite $\epsilon$-infimizer $\mathcal{\bar{X}}$ of \eqref{(P)} is called a \emph{finite (weak) $\epsilon$-solution} to \eqref{(P)} if it consists of only (weak) minimizers.
\end{definition}

Note that if $\bar{\mathcal{X}}$ is a finite (weak) $\epsilon$-solution, we have the following inner and outer approximation of the upper image
\begin{align*}
\conv \Gamma(\bar{\mathcal{X}})+C-\epsilon\{c\} \supseteq \mathcal{P} \supseteq \conv \Gamma(\bar{\mathcal{X}})+C.
\end{align*}

For some parameter vector $w \in \R^q$, the convex program 
\begin{align} \label{P1}  \tag{P$_1(w)$}
\min \left\{w^T\Gamma(x) : x \in X,\; g(x) \leq 0 \right\}  
\end{align}
is the well-known {\em weighted sum scalarization} of \eqref{(P)}. Using the Lagrangian 
\begin{equation*}
L_w: X \times\R^m \to \R,\qquad L_w(x,u):=w^T\Gamma(x)+u^Tg(x)	
\end{equation*}
we define the dual program
\begin{align} \label{D1}  \tag{D$_1(w)$}
\max \left\{\phi_w (u) : u \in \R^m_+ \right\}  
\end{align}
with the dual objective function $\phi_w(u):=\inf_{x\in X} L_w(x,u)$. We close this section with a well-known scalarization result, see e.g. \cite{jahn,luc}. 
\begin{proposition} \label{dualprop1}
Let $w \in C^+\setminus\{0\}$. An optimal solution $x^w$ of \eqref{P1} is a weak minimizer of \eqref{(P)}.
\end{proposition}

\subsection{Geometric Duality}

Geometric duality for vector optimization problems has first been introduced for LVOPs in \cite{geometric_lp}. The idea behind geometric duality in the linear case is to consider a dual problem with a lower image being dual to the upper image $\mathcal{P}$ in the sense of duality of convex polyhedra. Recently, Heyde \cite{geometric} extended geometric duality to the case of CVOPs. 

We next define the \emph{geometric dual problem} of \eqref{(P)}. Recall that $c\in \Int C$ is fixed. Further, we fix $c^1, \ldots, c^{q-1} \in\R^q$ such that the vectors $c^1, \ldots, c^{q-1}, c$ are linearly independent. Using the nonsingular matrix $T:=(c^1,\ldots, c^{q-1}, c)$, we define 
\begin{align*}
w(t):=\left(\left( t_1, \ldots, t_{q-1}, 1 \right) T^{-1}\right)^T
\end{align*}
for $t\in\R^q$. For arbitrary $w,t \in \R^q$ we immediately obtain the useful statement 
\begin{equation}\label{eq_wt}
 (w(t)= w,\; t_q=1) \; \iff \; (t=T^T w,\; c^T w=1).
\end{equation}
The geometric dual of \eqref{(P)} is given by
\begin{align*}\label{(D)}
\text{maximize~} D^*(t) \text{~~~~with respect to~~} \leq_K \text{~~subject to~}  w(t) \in C^+ \tag{D},
\end{align*}
where the objective function is 
$$  D^*(t) := (t_1,\ldots, t_{q-1},\inf_{x\in\mathcal{X}}\left[w(t)^T\Gamma(x)\right])^T,$$
 and the ordering cone is $K:= \R_+(0,0,\ldots,0,1)^T = \R_+e^q$. 

Similar to the upper image for the primal problem \eqref{(P)}, we define the \emph{lower image} for \eqref{(D)} as $\mathcal{D} := D^*(\mathcal{T})-K$, where $\mathcal{T} := \{t\in \R^q : w(t)\in C^+\}$ is the feasible region of \eqref{(D)}. The lower image can be written as
\begin{align*}
\mathcal{D} := \left\{ t\in\R^q : w(t) \in C^+, t_q \leq \inf_{x\in\mathcal{X}}\left[ w(t)^T\Gamma(x)\right]\right\}.
\end{align*}

\begin{proposition}\label{dualprop0} 
Let $t \in \mathcal{T}$ and $w := w(t)$. If \eqref{P1} has a finite optimal value $y^w \in \R$, then $D^*(t)$ is $K$-maximal in $\mathcal{D}$ and $D^*(t) = (t_1,\ldots,t_{q-1},y^w) \in \bd\mathcal{D}$.
\end{proposition} 
\begin{proof} 
Since $t\in\mathcal{T}$, we have $w(t)\in C^+\setminus\{0\}$. We obtain $D^*(t)\in \R^q$ and $D^*(t)$ is $K$-maximal in $\mathcal{D}$ since $w(t)$ does not depend on $t_q$. The remaining statements are obvious.
\end{proof}

\begin{remark}
Another possibility is to define the dual objective function as
$$ D^*(u,t) := (t_1,\ldots, t_{q-1},\inf_{x\in X} L_{w(t)}(x,u))^T.$$
By this variant the special structure of the feasible set is taken into account. As this definition leads to the same lower image $\mathcal{D}$, very similar results can be obtained.
\end{remark}

The lower image $\mathcal{D}$ is a closed convex set. To show convexity, let $t^1,t^2 \in \mathcal{D}$, $\alpha \in [0,1]$, and set $t:=\alpha t^1 + (1-\alpha) t^2$. Then, $w(t) = \alpha w(t^1)+(1-\alpha)w(t^2) \in C^+$, as $C^+$ is convex. It holds $t \in \mathcal{D}$ as we have
\begin{align*}
t_q = \alpha t^1_q + (1-\alpha)t^2_q &\leq \alpha\inf_{x\in\mathcal{X}}[w(t^1)^T\Gamma(x)]+ (1-\alpha)\inf_{x\in\mathcal{X}}[w(t^2)^T\Gamma(x)]\\
&\leq \inf_{x\in\mathcal{X}}[w(t)^T\Gamma(x)].
\end{align*}
Taking into account \cite[Proposition 5.5 and Remark 2]{geometric} we see that $\mathcal{D}$ is closed (since it can be expressed as $\mathcal{D}= -\epi f^*$ for some function $f$, cf. \cite{geometric}).

In \cite{geometric} it was proven that there is a one-to-one correspondence between the set of $K$-maximal exposed faces of $\mathcal{D}$, and the set of all weakly $C$-minimal exposed faces of $\mathcal{P}$. Recall that the $K$-maximal elements of $\mathcal{D}$ are defined as elements of the set $\MaxK(\mathcal{D}) : = \{ t \in\mathcal{D}: (\{t\}+K\setminus \{0\}) \cap\mathcal{D} = \emptyset\}$.
For $y, y^* \in \R^q$ we define:
\begin{align*}
\varphi: \R^q\times\R^q \rightarrow \R, \:\:\: & \varphi(y,y^*):= (y^*_1,\ldots, y^*_{q-1},1)T^{-1}y-y^*_q,\\
H:\R^q \rightrightarrows \R^q, \:\:\: & H(y^*) := \{y\in\R^q :\varphi(y,y^*)=0 \},\\
H^*:\R^q \rightrightarrows \R^q, \:\:\: & H^*(y) := \{y^*\in \R^q : \varphi(y,y^*)=0 \}.
\end{align*}
The duality map $\Psi: 2^{\R^q}\rightarrow 2^{\R^q}$ is constructed as
\begin{align*}
\Psi(F^*):= \bigcap_{y^*\in F^*}H(y^*)\cap\mathcal{P}.
\end{align*}
The following \emph{geometric duality theorem} states that $\Psi$ is a duality map between $\mathcal{P}$ and $\mathcal{D}$.

\begin{theorem}[{\cite[Theorem 5.6]{geometric}}]
$\Psi$ is an inclusion reversing one-to-one mapping between the set of all $K$-maximal exposed faces of $\mathcal{D}$ and the set of all weakly $C$-minimal exposed faces of $\mathcal{P}$. The inverse map is given by
\begin{align*}
\Psi^{-1}(F) = \bigcap_{y \in F}H^*(y)\cap\mathcal{D}.
\end{align*}
\end{theorem}

Similar to Definition~\ref{soln} and following the pattern of the linear case in \cite{lvop,lohne}, we introduce a solution concept for the dual problem \eqref{(D)}. 

\begin{definition} \label{dualsoln}
A point $\bar{t} \in \mathcal{T}$ is a maximizer for \eqref{(D)}, if there is no $t \in \mathcal{T}$ with $D^*(t)\geq_K D^*(\bar{t})$ and $D^*(t) \neq D^*(\bar{t})$, that is, $D^*(\bar{t})$ is a $K$-maximal element of $D^*(\mathcal{T})$. 
A nonempty set $\bar{\mathcal{T}} \subseteq\mathcal{T}$ is called a \emph{supremizer of \eqref{(D)}} if 
$\cl \conv (D^*(\bar{\mathcal{T}})-K) = \mathcal{D}$.
A supremizer $\bar{\mathcal{T}}$ of \eqref{(D)} is called a \emph{solution} to \eqref{(D)} if it consists of only maximizers.
\end{definition}

As for the primal problem, we also consider an $\epsilon$-solution of \eqref{(D)} consisting of only finitely many maximizers. This concept is an extension of $\epsilon$-solutions for LVOPs introduced in \cite[Remark 4.10]{lvop} to the convex setting.

\begin{definition}\label{weakdualsoln}
For the geometric dual problem \eqref{(D)}, a nonempty finite set $\bar{\mathcal{T}}\subseteq \mathcal{T}$ is called a \emph{finite $\epsilon$-supremizer of \eqref{(D)}} if 
\begin{equation}\label{eqweakdualsoln}
\conv D^*(\bar{\mathcal{T}})-K+\epsilon \{e^q\} \supseteq \mathcal{D}.
\end{equation}
 A finite $\epsilon$-supremizer $\bar{\mathcal{T}}$ of \eqref{(D)} is called a \emph{finite $\epsilon$-solution} to \eqref{(D)} if it consists of only maximizers.
\end{definition}

Note that if $\bar{\mathcal{T}}$ is a finite $\epsilon$-solution of \eqref{(D)}, one obtains the following inner and outer polyhedral approximation of the lower image
\begin{align*}
\conv D^*(\bar{\mathcal{T}})-K+\epsilon\{e^q\} \supseteq \mathcal{D} \supseteq \conv D^*(\bar{\mathcal{T}})-K.
\end{align*}

We next show that an approximation of the upper image of \eqref{(P)} can be used to obtain an approximation of the lower image of \eqref{(D)} and vise versa. We have the following duality relations, which will be used to prove the correctness of the algorithms, see  Theorems~\ref{thmprimal} and~\ref{thmdual} below.

\begin{proposition}\label{apprprop1}
Let $\emptyset \neq \bar{\mathcal{P}} \subsetneq \R^q$ be a closed and convex set such that $\bar{\mathcal{P}} = \bar{\mathcal{P}} + C$ and let $\bar{\mathcal{D}}$ be defined by
\begin{align}\label{Di}
\bar{\mathcal{D}} = \{y^*\in \R^q: \forall y \in \bar{\mathcal{P}}, \varphi(y,y^*) \geq 0\}.
\end{align}
Then,
\begin{align}\label{Po}
\bar{\mathcal{P}} = \{y\in \R^q: \forall y^* \in \bar{\mathcal{D}}, \varphi(y,y^*) \geq 0\}.
\end{align}
\end{proposition}
\begin{proof}
The inclusion $\subseteq$ is obvious. Assume that the inclusion $\supseteq$ does not hold. Then there exists $\bar y \in \R^q\setminus \bar{\mathcal{P}}$ such that $\varphi(\bar y,y^*)\geq 0$ for all $y^* \in \bar{\mathcal{D}}$. By usual separation arguments, we get $\eta \in C^+\smz$ with $\eta ^T \bar y < \inf_{y \in \bar{\mathcal{P}}} \eta^T y =: \gamma$. Since $c \in \Int C$ we can assume $\eta^T c = 1$. Setting $\bar y^*:=(\eta_1,\dots,\eta_{q-1},\gamma) T$, we get $\varphi(y,\bar y^*)=\eta^T y-\gamma$. For all $y \in \bar{\mathcal{P}}$, we have $\eta^T y-\gamma \geq 0$, i.e., $\bar y^* \in \bar{\mathcal{D}}$. But $\varphi(\bar y,\bar y^*)=\eta^T \bar y-\gamma< 0$, a contradiction.
\end{proof}

\begin{proposition}\label{apprprop2}
Let $\emptyset \neq \bar{\mathcal{D}} \subsetneq \R^q$ be a closed and convex set such that $\bar{\mathcal{D}} = \bar{\mathcal{D}} - K$, $\bar{\mathcal{D}} \neq \bar{\mathcal{D}} + K$ and let $\bar{\mathcal{P}}$ be defined by (\ref{Po}). Then (\ref{Di}) holds.
\end{proposition}
\begin{proof}
The inclusion $\subseteq$ is obvious. Assume the inclusion $\supseteq$ does not hold. Then, there exists $\bar{y}^* \in \R^q\setminus \bar{\mathcal{D}}$ with $\varphi(y,\bar{y}^*)\geq 0$ for all $y \in \bar{\mathcal{P}}$. Applying a separation argument, we obtain $\eta \in K^+\setminus\{0\}$ with $\eta^T\bar{y}^* > \sup_{y^* \in \bar{\mathcal{D}}} \eta^Ty^* =: \gamma(\eta)$. Using the assumption $\bar{\mathcal{D}} \neq \bar{\mathcal{D}} + K$, we get $\gamma(e^q)<\infty$. Set $\alpha(\eta):=\eta^T\bar y^* - \gamma(\eta)>0$ and let $\beta > 0$ such that 
$\beta(\gamma(e^q)-\bar y^*_q) < \alpha(\eta)$. For $\bar \eta:= \eta + \beta e^q \in \Int K^+$ we have
$$ \bar\eta^T \bar y^* = \eta^T\bar y^* + \beta \bar y^*_q  > \eta^T\bar y^* - \alpha(\eta) + \beta \gamma(e^q)
=\gamma(\eta) + \beta \gamma(e^q) \geq \gamma(\bar \eta). $$ 
Without loss of generality we can assume that $\bar \eta_q=1$. Setting $\bar{y} = T(-\bar \eta_1, \ldots, -\bar \eta_{q-1},\gamma(\bar \eta))^T$, we have $\varphi(\bar{y},y^*) = \gamma(\bar\eta)-\bar\eta^Ty^* \geq 0$ for all $y^*\in \bar{\mathcal{D}}$ which implies $\bar{y}\in \bar{\mathcal{P}}$. But, $\varphi(\bar{y},\bar{y}^*) = \gamma(\bar \eta)-\bar\eta^T\bar{y}^* < 0$, a contradiction.
\end{proof}
Both duality relations \eqref{Di} and \eqref{Po} hold for the upper and lower images of \eqref{(P)} and \eqref{(D)}.

\begin{proposition}\label{strongdual} Equations \eqref{Di} and \eqref{Po} are satisfied for $\bar{\mathcal{P}}=\mathcal{P}$ and $\bar{\mathcal{D}}=\mathcal{D}$.
\end{proposition}
\begin{proof} We have $\varphi(y,y^*)\geq 0$ for $y \in \mathcal{P}$ and $y^* \in \mathcal{D}$, which implies  $\subseteq$ in \eqref{Di}. Let $y^* \in \R^q$ such that $\varphi(y,y^*)=w(y^*)^T y-y^*_q\geq 0$ for all $y \in \mathcal{P}$. Then, $D^*_q(y^*) \geq y^*_q$. Using $\mathcal{P} + C = \mathcal{P}$, we get $w(y^*)\in C^+$. Hence $y^*\in \mathcal{D}$, i.e., \eqref{Di} holds. Proposition \ref{apprprop1} yields \eqref{Po}. 
\end{proof}

\section{Algorithms for CVOPs}\label{secalgo}

Let us consider the convex vector optimization problem \eqref{(P)} with polyhedral ordering cones $C$ and $D$. Then, we can assume without loss of generality that $D = \R^m_+ = \{y \in \R^m: y_1 \geq 0,\ldots, y_m \geq 0\}$, which means that $g$ is component-wise convex. Indeed, whenever $D$ is polyhedral, the feasible set $\mathcal{X} = \{x\in X:g(x)\leq_D 0\}$ can be written as $\{x\in X: h(x)=(h_1(x),\ldots,h_l(x))\leq_{\R^l_+} 0 \}$, where $d_1,\ldots,d_l \in \R^m$ are the $l$ generating vectors of the dual cone $D^+$ of $D$, and $h: X \rightarrow \R^l$ is defined by $h_i(x):=d_i^Tg(x)$ for $i = 1,\ldots,l$. Moreover, $g$ is $D$-convex if and only if $h$ is $\R^l_+$-convex. 

In addition to the assumptions made in the problem formulation of \eqref{(P)} in the beginning of Section~\ref{subsetP}, we will assume the following throughout the rest of the paper.
\begin{assumption}\label{assmp} Let the following hold true.
\begin{enumerate} [(a)]
\item The feasible region $\mathcal{X}$ is a compact subset of $\R^{n}$.
\item $\mathcal{X}$ has non-empty interior.
\item The objective function $\Gamma: X \rightarrow \R^q$ is continuous.
\item The ordering cone $C$ is polyhedral, and $D = \R^m_+$. 
\end{enumerate}
\end{assumption}

Assumption~\ref{assmp} implies that problem \eqref{(P)} is bounded. Indeed, as $\mathcal{X}$ is compact and $\Gamma$ is continuous, $\Gamma(\mathcal{X})$ is compact, in particular, bounded. Thus, there exists $a\in\R^q$ and $r>0$ such that the open ball $B(a,r)$ around $a$ with radius $r$ contains $\Gamma(\mathcal{X})$. Furthermore, as $c \in \Int C$, there exists $\rho > 0$ such that $B(a,\rho) \subseteq a - c + C$. Now one can check that $y := a - \frac{r}{\rho}c$ satisfies $\{y\} + C \supseteq B(a,r) \supseteq \Gamma(\mathcal{X})$.

Another consequence of Assumption~\ref{assmp} is that $\Gamma(\mathcal{X})+C$ is closed, i.e. the upper image $\mathcal{P}$ as defined in \eqref{upim1} can be expressed as $\mathcal{P}= \Gamma(\mathcal{X})+C$. Assumption~\ref{assmp} guaranties the existence of solutions and finite $\epsilon$-solutions to \eqref{(P)}.

\begin{proposition}\label{exsol}
Under Assumption~\ref{assmp}, a solution to \eqref{(P)} exists.
\end{proposition}
\begin{proof}
This is a consequence of the vectorial Weierstrass Theorem (\cite[Theorem 6.2]{freshlook} or \cite[Theorem 2.40]{lohne}).
A solution in the sense of Definition \ref{soln} corresponds to a `mild convexity solution' in \cite[Definition 2.48]{lohne}, see \cite[Propositions 1.58, 1.59]{lohne}. The vectorial Weierstrass Theorem \cite[Theorem 2.40]{lohne} implies the existence of a `solution' in the sense of \cite[Definition 2.20]{lohne}, obviously being a `mild solution' in the sense of \cite[Definition 2.41]{lohne} and being a `mild convexity solution' by \cite[Corollary 2.51]{lohne}.
\end{proof}

\begin{proposition}\label{existence_soln}
Under Assumptions~\ref{assmp}, for any $\epsilon > 0$, there exists a finite $\epsilon$-solution to problem \eqref{(P)}.
\end{proposition}
\begin{proof}
As already mentioned in the proof of Proposition \ref{exsol}, there exists a `solution' $\bar{\mathcal{X}}$ in the sense of \cite[Definition 2.41]{lohne}, that is, $\bar{\mathcal{X}}\neq \emptyset$ is the set of all minimizers for \eqref{(P)} and we have $\cl(\Gamma(\bar{\mathcal{X}})+C) = \cl(\Gamma(\mathcal{X})+C)$. For arbitrary fixed $\epsilon>0$,  $\{\Gamma(x) - \epsilon c + \Int C: x \in \bar{\mathcal{X}}\}$ is an open cover of $\Gamma(\mathcal{X})$, which is a compact set by Assumptions~\ref{assmp}. Hence there is a finite subcover $\{\Gamma(x) - \epsilon c + \Int C: x \in \hat{\mathcal{X}}\}$ and $\hat{\mathcal{X}}$ is a finite $\epsilon$-solution.
\end{proof}

\subsection{Primal Algorithm}
Benson's algorithm has been extended to approximate the upper image $\mathcal{P}$ of a convex vector optimization problem in \cite{ehrgott}. In this section, we will generalize and simplify this algorithm as detailed in the introduction. 
The algorithm can be explained as follows. Start with an initial outer approximation $\mathcal{P}_0$ of $\mathcal{P}$ and compute iteratively a sequence $\mathcal{P}_0 \supseteq \mathcal{P}_1 \supseteq \mathcal{P}_2 \supseteq ... \supseteq \mathcal{P}$ of better outer approximations.

The first step in the $k^{\text{th}}$ iteration is to compute the vertices of $\mathcal{P}_k$. During the algorithm, whenever $\mathcal{P}_k$ is updated, it is given by an H-representation. To convert an H-representation into a V-representation (and vise versa) one uses {\em vertex enumeration}, see e.g. \cite{BreFukMar98}. For a vertex $v$ of $\mathcal{P}_k$, a point $y$ on the boundary of the upper image, which is in `minimum distance' to $v$, is determined. Note that  $y = \Gamma(x)+c$ for some $x\in\mathcal{X}$, and $c\in C$. We add all those $x$ to a set $\bar{\mathcal{X}}$,
where $\bar{\mathcal{X}}$ has to be initialized appropriately. This set will be shown to be a finite weak $\epsilon$-solution to \eqref{(P)} at termination. If the minimum distance is less than or equal to an error level $\epsilon > 0$, which is determined by the user, the algorithm proceeds to check another vertex of $\mathcal{P}_k$ in a similar way. If the minimum distance is greater than the error level $\epsilon$, a cutting plane, i.e., a supporting hyperplane of the upper image $\mathcal{P}$ at the point $y$ and its corresponding halfspace $H_k$ containing the upper image are calculated. The new approximation is obtained as $\mathcal{P}_{k+1}= \mathcal{P}_k \cap H_k$. The algorithm continues in the same manner, until all vertices of the current approximation are in `$\epsilon$-distance' to the upper image. The details of the algorithm are explained below.

To compute $\mathcal{P}_0$, let $Z$ be the matrix, whose columns  $z^1,\ldots,z^J$ are the generating vectors of the dual cone $C^+$ of the ordering cone $C$ and let $z^j$ be normalized in the sense that $c^Tz^j = 1$ for all $j = 1,\ldots, J$ (recall that $c \in \Int C$ is fixed). Denote $x^j \in X$ ($i=1,\dots,J$) the optimal solutions of (P$_1(z^j)$), which always exist by Assumptions \ref{assmp}~(a) and (c).
Define the halfspace 
$$H_j := \{y\in \R^q : (z^j)^Ty \geq (z^j)^T \Gamma(x^{j})\}.$$ 
Note that $t^j:=T^T z^j$ belongs to the feasible set $\mathcal{T}$ of \eqref{(D)} since $c^Tz^j = 1$ and $w(t^j)=z^j \in C^+$, compare \eqref{eq_wt}. We have $D^*(t^j) = (t^j_1, \ldots,t^j_{q-1}, (z^j)^T\Gamma(x^j))$, which implies
\begin{align} \label{HJ}
H_j = \{y\in \R^q :  \varphi(y,D^*(T^T z^j))\geq 0 \}.
\end{align}
It is easy to check that for all $j$, $H_j$ contains the upper image $\mathcal{P}$ (`weak duality'). We define the initial outer approximation as the intersection of these halfspaces, that is, 
\begin{align} \label{initialization}
\mathcal{P}_0 := \bigcap_{j=1}^J H_j. 
\end{align}
Since $C$ is pointed and \eqref{(P)} is bounded, $\mathcal{P}_0$ contains no lines. By \cite[Corollary 18.5.3]{rockafellar} we conclude that $\mathcal{P}_0$ has at least one vertex.
Vertex enumeration yields the set of all vertices. 

We will use the following convex program which depends on a parameter vector $v \in \R^q$, which typically does not belong to $\Int \mathcal{P}$,
\begin{align*}\label{P2}
\min\left\{z \in \R:\;\; g(x) \leq 0, \;\; Z^T(\Gamma(x) - zc - v) \leq 0 \right\}. \tag{P$_2(v)$}
\end{align*}
The second part of constraints can be expressed as $\Gamma(x) - zc - v\in -C$. Hence, the Lagrangian 
\begin{equation}\label{Lagrange2}
	L_v:(X \times\R)\times(\R^m\times\R^q)\to \R,\;\; L_v(x,z,u,w):=z+u^Tg(x)+w^T\Gamma(x) - w^T c z - w^T v,
\end{equation}
yields the dual problem
\begin{align*}
\max\left\{ \inf_{x \in X, z \in \R} L_v(x,z,u,w):\;\; u \geq 0, w \in C^+ \right\},
\end{align*}
which can be equivalently expressed as
\begin{align*}\label{D2}
\max\left\{\inf_{x\in X}\{u^Tg(x)+w^T\Gamma(x)\}-w^Tv : \;\; u\geq 0,\; w^Tc = 1,\; Y^Tw \geq 0 \right\}, \tag{D$_2(v)$}
\end{align*}
where $Y$ is the matrix whose columns are the generating vectors of the cone $C$.

The following propositions will be used later to prove the correctness of the algorithm.
\begin{proposition}\label{primalprop1}
For every $v \in \R^q$, there exist optimal solutions $(x^v, z^v)$ and $(u^v,w^v)$ to problems \eqref{P2} and \eqref{D2}, respectively, and the optimal values coincide.
\end{proposition}
\begin{proof} $\mathcal{X}$ is compact by Assumption \ref{assmp}~(a). The set $\mathcal{X}_2:=\{x\in X : Z^T(\Gamma(x)-zc-v) \leq 0\}$ is closed as $\Gamma$ is continuous by Assumption \ref{assmp}~(c). Thus the feasible set $\mathcal{X} \cap \mathcal{X}_2$ for \eqref{P2} is compact, which implies the existence of an optimal solution $(x^v,z^v)$ of \eqref{P2}. By Assumption \ref{assmp}~(b) there exists $x^0 \in X$ with $g(x^0)<0$. Since $c\in \Int C$, we have $Z^Tc>0$. Taking $z^0$ large enough, we obtain $Z^Tv+z^0 Z^T c > Z^T\Gamma(x^0)$. Hence, $(x^0,z^0)$ satisfies Slater's condition. Convex programming duality implies the existence of a solution $(u^v,w^v)$ of \eqref{D2} and coincidence of the optimal values.
\end{proof}

\begin{proposition}\label{primalprop2}
Let $(x^v, z^v)$  be an optimal solution of \eqref{P2} for $v \in \R^q$. Then, $x^v$ is a weak minimizer of \eqref{(P)}, and $y^v := v + z^v c \in \wMinc(\mathcal{P})$. Moreover, $v \in \wMinc(\mathcal{P})$ if and only if $z^v = 0$.
\end{proposition}
\begin{proof} Suppose $x^v$ is not a weak minimizer of \eqref{(P)}, i.e., $\Gamma(\bar{x}) <_C \Gamma(x^v)$ for some $\bar{x} \in \mathcal{X}$. We have $\Gamma(\bar x) = \Gamma(x^v) -\bar c$ for some $\bar c \in \Int C$ and there exists $\varepsilon > 0$ such that $\bar c - \varepsilon c \in C$, hence
$$ \Gamma(\bar x) = \Gamma(x^v) -\bar c \leq_C \Gamma(x^v) -\varepsilon c \leq_C (z^v-\varepsilon)c  + v.$$
Multiplying by the matrix $Z$ whose columns are the generating vectors of $C^+$, we get $Z^T \Gamma(\bar x) \leq Z^T((z^v-\varepsilon)c  + v)$, which implies that $(\bar x,z^v-\varepsilon)$ is feasible for \eqref{P2} but generates a smaller value than the optimal value $z^v$, a contradiction.

To show that $y^v \in \wMinc(\mathcal{P})$, first note that $Z^T(z^vc+v-\Gamma(x^v)) \geq 0$. Since $C$ is given by $Z$ as $C=\{y \in \R^q : Z^T y \geq 0 \}$, we have $z^v c+v-\Gamma(x^v) \in C$, i.e., $\Gamma(x^v) \leq_C y^v$ and thus $y^v \in \mathcal{P}$. Suppose that $y^v \in \mathcal{P} + \Int C$. Then $v + z^v c \in \mathcal{P} + \Int C = \Int \mathcal{P}$ (see e.g. \cite[Corollary 1.48 (iii)]{lohne} for the last equation). There exists $\varepsilon>0$ with $v + (z^v-\varepsilon) c \in \mathcal{P}$, i.e., there is $\bar x \in \mathcal{X}$ with $v + (z^v-\varepsilon) c \in \Gamma(\bar x)+C$. This means that $(\bar x,z^v-\varepsilon)$ is feasible for \eqref{P2} and has a smaller value than the optimal value $z^v$, a contradiction.
\end{proof}

\begin{proposition}\label{primalprop3}
Let  $v \in \R^q$ and let $(x^v, z^v)$ and $(u^v,w^v)$ be optimal solutions to \eqref{P2} and \eqref{D2}, respectively. Then, $t^v:=T^T w^v$ is a maximizer for \eqref{(D)} and 
\begin{align}\label{eqdualprimal}
D^*_q (t^v) = (w^v)^T\Gamma(x^v) + (u^v)^Tg(x^v) = (w^v)^Tv+z^v.
\end{align}
\end{proposition}
\begin{proof} As $(u^v,w^v)$ is feasible for \eqref{D2}, we have $w^v \in C^+$ and $c^Tw^v=1$. By \eqref{eq_wt}, we obtain $w^v=w(t^v) \in C^+$. Hence $t^v$ is a maximizer of \eqref{(D)} by Proposition \ref{dualprop0} and the fact that \eqref{P1} has an optimal solution by Assumption~\ref{assmp}. It remains to show (\ref{eqdualprimal}). Since $(x^v,z^v,u^v,w^v)$ is a saddle point of the Lagrangian $L_v$ in  \eqref{Lagrange2}, and taking into account that $c^T w^v=1$, we get
$$ z^v =  (u^v)^Tg(x^v)+(w^v)^T\Gamma(x^v) - (w^v)^T v = 
          \inf_{x\in X}\left\{(u^v)^T g(x) + (w^v)^T \Gamma(x) \right\} - (w^v)^T v.$$
This yields \eqref{eqdualprimal} if we can show that
$$ \inf_{x\in\mathcal{X}}(w^v)^T\Gamma(x) = \inf_{x\in X}\left\{(u^v)^T g(x) + (w^v)^T \Gamma(x) \right\}.$$
But, if $w^v$ is considered to be a parameter, $u^v$ is an optimal solution for (D$_1(w^v)$) and the desired statement follows from strong duality between (P$_1(w^v)$) and (D$_1(w^v)$). 
\end{proof}

The next step is to show that a supporting hyperplane of $\mathcal{P}$ at $y^v = v + z^vc$ can be found using an optimal solution of \eqref{D2}.

\begin{proposition}\label{primalprop4}
Let $v \in \R^q$ and let $(x^v,z^v)$ and $(u^v,w^v)$ be optimal solutions for \eqref{P2} and \eqref{D2}, respectively. For $t^v:=T^T w^v$, $\mathcal{H}:= H(D^*(t^v)) = \{y\in\R^q: \varphi(y,D^*(t^v))= 0\}$ is a supporting hyperplane of $\mathcal{P}$ at $y^v = v + z^vc$.
\end{proposition}
\begin{proof} 
By Proposition~\ref{primalprop2} and \eqref{eq_wt}, $y^v \in \mathcal{P} \cap \mathcal{H}$. Since $w(D^*(t^v))= w(t^v)\neq 0$, $\mathcal{H}$ is a hyperplane. Let $y\in \mathcal{P}$ and let $(\bar x,\bar z)$ be feasible for (P$_2(y)$). Then, we have $\bar z\leq 0$. Of course, $(u^v,w^v)$ is feasible for (D$_2(y)$). Using weak duality for (P$_2(y)$)/(D$_2(y)$) and strong duality for \eqref{P2}/\eqref{D2}, we obtain 
$$0\geq \bar z \geq \inf_{x\in X}\left\{ (u^v)^T g(x) + (w^v)^T \Gamma(x)\right\} -(w^v)^T v + (w^v)^T (v-y)=z^v+(w^v)^T(v-y).$$ 
Using \eqref{eqdualprimal} and \eqref{eq_wt} we conclude that $\varphi(y,D^*(t^v)) = (w^v)^Ty-D^*_q(t^v) \geq 0$.
\end{proof}

\begin{proposition}\label{approximationremark}
For $\epsilon>0$, let $\bar{\mathcal{X}}$ be a finite (weak) $\epsilon$-solution of \eqref{(P)}, and define $\mathcal{P}_{\epsilon} := 
\conv \Gamma(\bar{\mathcal{X}})+C-\epsilon\{c\}$. Then, ${\mathcal{D}}_{\epsilon} := \{y^*\in \R^q: \varphi(y,y^*) \geq 0, \forall y \in \mathcal{P}_{\epsilon}\}$ is an inner $\epsilon$-approximation of the lower image $\mathcal{D}$, that is, ${\mathcal{D}}_\epsilon + \epsilon \{e^q\} \supseteq \mathcal{D} \supseteq {\mathcal{D}}_{\epsilon}$. 

Similarly, let $\bar{\mathcal{T}}$ be a finite $\epsilon$-solution of \eqref{(D)}, and define $\mathcal{D}_{\epsilon} := \conv D^*(\bar{\mathcal{T}})-K+\epsilon \{e^q\}$. Then, ${\mathcal{P}}_{\epsilon} := \{y\in \R^q: \varphi(y,y^*) \geq 0, \forall y^* \in \mathcal{D}_{\epsilon}\}$ is an inner $\epsilon$-approximation of the upper image $\mathcal{P}$, that is,  ${\mathcal{P}}_\epsilon - \epsilon \{c\} \supseteq\mathcal{P} \supseteq{\mathcal{P}}_{\epsilon}$.
\end{proposition}
\begin{proof} 
In the first statement $\mathcal{P}_{\epsilon}$ is an outer approximation of the upper image $\mathcal{P}$. Using \eqref{eq_wt}, we obtain that $\varphi(y,y^*+ze^q) = \varphi(y-zc,y^*)$ for arbitrary $y,y^* \in \R^q$, $z\in \R$. Now it is straightforward to show that ${\mathcal{D}}_\epsilon + \epsilon \{e^q\} \supseteq \mathcal{D} \supseteq {\mathcal{D}}_{\epsilon}$, where Proposition \ref{strongdual} is useful. The second statement can be proven similarly. 
\end{proof}

We are now ready to present the Primal Approximation Algorithm to solve \eqref{(P)} and \eqref{(D)}. 

\begin{algorithm}
\caption{Primal Approximation Algorithm for \eqref{(P)} and \eqref{(D)}}
\begin{algorithmic}[1] \label{alg_1}
\STATE Compute optimal solutions $x^j$ of (P$_1(z^j)$) for $j=1,\ldots,J$;
\STATE Store an H-representation $\mathcal{P}^H$ of $\mathcal{P}_0$ according to \eqref{initialization};
\STATE $k \leftarrow 0$; $\bar{\mathcal{X}} \leftarrow \{x^1,\ldots, x^J\}$; $\bar{\mathcal{T}} \leftarrow \{T^T z^1, \ldots, T^T z^J\}$; 
\STATE $D\leftarrow D^*(\bar{\mathcal{T}})$, where $D^*_q(T^T z^j) = (z^j)^T \Gamma(x^j)$, $j=1,\dots,J$; 
\REPEAT 
\STATE $M \leftarrow \R^q$;
\STATE Compute the set $\mathcal{P}^V$  of vertices of $\mathcal{P}_k$ from its H-representation $\mathcal{P}^H$;
	\FOR{$i=1:\abs{\mathcal{P}^V}$}
		\STATE Let $v$ be the $i^{th}$ element of $\mathcal{P}^V$ (i.e. the $i^{th}$ vertex of $\mathcal{P}_k$);
		\STATE Compute optimal solutions $(x^v, z^v)$ to \eqref{P2} and $(u^v, w^v)$ to \eqref{D2};
		\STATE $\bar{\mathcal{X}} \gets \bar{\mathcal{X}} \cup \{x^v\}$; $\bar{\mathcal{T}}\gets \bar{\mathcal{T}} \cup \{T^T w^v\}$; Update $D$ using \eqref{eqdualprimal} such that $D=D^*(\bar{\mathcal{T}})$;		
		\IF{$z^v > \epsilon$}	
 			\STATE $M \gets M \cap \{ y \in \R^q : \varphi(y,D^*(T^T w^v))\geq 0 \}$; 		
			\STATE break; (optional) 
		\ENDIF
	\ENDFOR
\IF{$M\neq\R^q$}
\STATE Store in $\mathcal{P}^H$ an H-representation of $\mathcal{P}_{k+1}=\mathcal{P}_k \cap M$ and set $k \gets k+1$; 
\ENDIF
\UNTIL $M=\R^q$
\STATE Compute the vertices $\mathcal{V}$ of $\{y\in\R^q: \varphi(y,y^*)\geq 0, \forall y^* \in D=D^*(\bar{\mathcal{T} } )\}$;
\RETURN $\left\{ \begin{array}{ll}
\bar{\mathcal{X}}&: \text{A finite weak~}\epsilon \text{-solution to }\eqref{(P)}; \\
\bar{\mathcal{T}}&: \text{A finite~}\epsilon \text{-solution to }\eqref{(D)};\\
\mathcal{V} &:  \text{Vertices of an outer~}\epsilon \text{-approximation of ~} \mathcal{P};\\
\Gamma(\bar{\mathcal{X}}) &: \text{Vertices of an inner~}\epsilon \text{-approximation of ~} \mathcal{P}.\\
\end{array} \right.$
\end{algorithmic}
\end{algorithm}

\begin{theorem}\label{thmprimal}
Under Assumption~\ref{assmp} Algorithm~\ref{alg_1} works correctly: If the algorithm terminates, it returns a finite weak $\epsilon$-solution $\bar{\mathcal{X}}$ to \eqref{(P)}, and a finite $\epsilon$-solution $\bar{\mathcal{T}}$ to \eqref{(D)}.
\end{theorem}
\begin{proof}
By Assumption~\ref{assmp}, each problem (P$_1(z^j)$) in line 1 has an optimal solution $x^j$. By Propositions~\ref{dualprop1} and~\ref{dualprop0}, $x^j$ is a weak minimizer of \eqref{(P)} and $T^T z^j$ is a maximizer of \eqref{(D)}. Thus, the sets $\bar{\mathcal{X}}$ and $\bar{\mathcal{T}}$ are initialized by weak minimizers of \eqref{(P)} and maximizers of \eqref{(D)}, respectively. As noticed after \eqref{initialization}, $\mathcal{P}_0$ has at least one vertex. Thus, the set $\mathcal{P}^V$ in line 7 is nonempty. By Proposition \ref{primalprop1}, optimal solutions to \eqref{P2} and \eqref{D2} exist. By Proposition~\ref{primalprop2} a weak minimizer of (P) is added to $\bar{\mathcal{X}}$ in line~11. Proposition~\ref{primalprop3} ensures that a maximizer of \eqref{(D)} is added to $\bar{\mathcal{T}}$ in line~11. 
By Proposition~\ref{primalprop4}, we know that $M$ defined in lines~6 and~13 satisfies $M \supseteq \mathcal{P}$. This ensures that $\mathcal{P}_k \supseteq \mathcal{P}$ holds throughout the algorithm. By the same argument as used for $\mathcal{P}_0$, we know that $\mathcal{P}_k$ has at least one vertex. If the optional break in line~14 is in use, the inner loop (lines~8-16) is left if $z^v > \epsilon$, and the current outer approximation is updated by setting $\mathcal{P}_{k+1} = \mathcal{P}_k \cap M$. For the case without the optional break, see Remark~\ref{optional_break} below. The algorithm stops if $z^v \leq \epsilon$ for all the vertices $v$ of the current outer approximation $\mathcal{P}_k$. Let us assume this is the case after $\hat{k}$ iterations. We have 
\begin{align}\label{finer_app} 
	\mathcal{P}\subseteq \{y\in\R^q: \forall y^* \in D^*(\bar{\mathcal{T}}),\, \varphi(y,y^*)\geq 0\}  \subseteq \mathcal{P}_{\hat{k}}.
\end{align}
Indeed, the first inclusion follows from Proposition \ref{primalprop4} and the second inclusion follows from the construction of the set $\mathcal{P}_{\hat{k}}$, see \eqref{HJ}, \eqref{initialization} and lines 13 and 18. 

We next show that $\bar{\mathcal{X}}$ is a finite weak $\epsilon$-solution of \eqref{(P)}. We know that $\bar{\mathcal{X}}$ is finite and consists of weak minimizers only. We have to show \eqref{eqweaksoln}. We know that $\mathcal{P} \subseteq\mathcal{P}_{\hat{k}}$ and we have $\mathcal{P}_{\hat{k}} \subseteq \conv \Gamma(\bar{\mathcal{X}}) -\epsilon \{c\} + C =:\mathcal{P}_{\epsilon}$, since $z^v \leq \epsilon$ for each vertex $v$ of $\mathcal{P}_{\hat{k}}$.

Finally we show that $\bar{\mathcal{T}}$ is a finite $\epsilon$-solution of \eqref{(D)}. Clearly, $\bar{\mathcal{T}}$ is nonempty and finite and, as shown above, it consists of maximizers of \eqref{(D)} only. It remains to show \eqref{eqweakdualsoln}. Setting 
$\bar{\mathcal{D}}:=\conv D^*(\bar{\mathcal{T}})-K$, we have
$$\bar{\mathcal{P}}:=\{y\in\R^q: \forall y^*\in \bar{\mathcal{D}},\, \varphi(y,y^*)\geq 0\} \subseteq\{y\in\R^q: \forall y^* \in D^*(\bar{\mathcal{T} }),\, \varphi(y,y^*)\geq 0\}\subseteq \mathcal{P}_{\hat{k}} \subseteq \mathcal{P}_{\epsilon}.$$
Using Proposition~\ref{apprprop2}, we conclude 
$$\bar{\mathcal{D}} =  \{y^*\in \R^q: \forall y \in \bar{\mathcal{P}},\, \varphi(y,y^*) \geq 0\} \supseteq \{y^*\in \R^q: \forall y \in \mathcal{P}_{\epsilon},\, \varphi(y,y^*) \geq 0\} =: \mathcal{D}_{\epsilon}.$$
By Proposition~\ref{approximationremark}, ${\mathcal{D}}_\epsilon + \epsilon \{e^q\} \supseteq \mathcal{D}$. Altogether we have $\conv D^*(\bar{\mathcal{T}})-K + \epsilon \{e^q\} \supseteq \mathcal{D}$.
\end{proof}

Note that in general, $\bar{\mathcal{T}}$ produces a finer outer approximation of the upper image than $\mathcal{P}_{\hat{k}}$, see \eqref{finer_app}. The reason is that, in contrast to $\bar{\mathcal{T}}$, $M$ is not necessarily updated in each iteration.

\begin{remark}\label{optional_break}
The `break' in line 14 of Algorithm~\ref{alg_1} is optional. The algorithm with the break updates the outer approximation right after it detects a vertex $v$ with $z^v>\epsilon$. The algorithm without the break goes over all the vertices of the current outer approximation before updating it. In general, one expects a larger number of vertex enumerations for the first variant (with break), and more optimization problems to solve for the second one (without break). 
\end{remark}

\begin{remark}[Alternative Algorithm~\ref{alg_1}] \label{alternative}
We will now discuss a modification of Algorithm~\ref{alg_1}. It produces $\epsilon$-approximations of $\mathcal{P}$ and $\mathcal{D}$ as well, but with fewer vertices. Thus, the approximations are coarser. This alternative consists of three modifications. 

First, set $\bar{\mathcal{X}} =\emptyset$ in line~3 of Algorithm~\ref{alg_1}. Second, replace lines~11-15 of Algorithm~\ref{alg_1} with the alternative lines given below.
Then, $x^v$ is only added to the set $\bar{\mathcal{X}}$ if there is no cut, while the dual counterpart $T^T w^v$ is only added to $\bar{\mathcal{T}}$ if there is a cut.
Third, line~21 of Algorithm~\ref{alg_1} can be skipped and the vertices $\mathcal{P}^V$ of $\mathcal{P}_{\hat{k}}$ can be returned as the final outer $\epsilon$-approximation of $\mathcal P$. Note that the first two modifications imply $\mathcal{P}_{\hat{k}}=\{y\in\R^q: \varphi(y,y^*)\geq 0, \forall y^* \in D^*(\bar{\mathcal{T} } ) \}$, which makes line~21 superfluous.
 
Under these three modifications, one still finds a finite weak $\epsilon$-solution $\bar{\mathcal{X}}$ to \eqref{(P)}, and a finite $\epsilon$-solution $\bar{\mathcal{T}}$ to \eqref{(D)}; but, in general, $\bar{\mathcal{X}}$ and $\bar{\mathcal{T}}$ have less elements compared to the original version of Algorithm~\ref{alg_1}, so the approximation is coarser. This  variant is used in \cite{ehrgott}. We propose Algorithm~\ref{alg_1} as it yields a finer approximation at no additional cost.
\end{remark}
\begin{algorithm}
\caption*{Alternative to lines 11-15 of Algorithm~\ref{alg_1}}
\begin{algorithmic}
\IF{$z^v \leq \epsilon$}
	\STATE $\bar{\mathcal{X}} \gets \bar{\mathcal{X}} \cup \{x^v\}$;
\ELSE
	\STATE $\bar{\mathcal{T}}\gets \bar{\mathcal{T}} \cup \{T^T w^v\}$; Update $D$ using \eqref{eqdualprimal} such that $D=D^*(\bar{\mathcal{T}})$;
	\STATE $M \gets M \cap \{ y \in \R^q : \varphi(y,D^*(T^T w^v))\geq 0 \}$;
	\STATE break; (optional) 
\ENDIF
\end{algorithmic}
\end{algorithm}

\subsection{Dual Algorithm}

A dual variant of Benson's algorithm for LVOPs based on geometric duality \cite{geometric_lp} has been introduced in \cite{ehr_dual}. An extension which approximately solves dual LVOPs was established in \cite{dualalgorithm}. The main idea is to construct approximating polyhedra of the lower image $\mathcal{D}$ of the geometric dual problem \eqref{(D)}  analogous to Algorithm~\ref{alg_1}. Geometric duality is used to recover approximations of the upper image $\mathcal{P}$ of the primal problem \eqref{(P)}. 

We employ the same idea in order to construct a dual variant of an approximation algorithm for CVOPs. The algorithm starts with an initial outer approximation $\mathcal{D}_0$ of $\mathcal{D}$ and computes iteratively a sequence 
$\mathcal{D}_0 \supseteq \mathcal{D}_1 \supseteq \mathcal{D}_2 \supseteq ... \supseteq \mathcal{D}$ of smaller outer approximations. As in the primal algorithm, the vertices of $\mathcal{D}_k$ are found using vertex enumeration. Each vertex $t$ is added to the set $\bar{\mathcal{T}}$, which will be shown to be a finite $\epsilon$-solution to \eqref{(D)}. Then, we check the `distance' between $t$ and the boundary of the lower image. If it is greater than $\epsilon$, a point $\hat{t} \in \bd \mathcal{D}$, and a supporting hyperplane to $\mathcal{D}$ at $\hat{t}$ are determined. The approximation for the next iteration $\mathcal{D}_{k+1}$ is updated as the intersection of  $\mathcal{D}_k$ and the corresponding halfspace containing the lower image. The algorithm continues in the same manner until all vertices of the current approximation are in `$\epsilon$ distance' to the lower image. After this, it returns a finite weak $\epsilon$-solution $\bar{\mathcal{X}}$ of (P), a finite $\epsilon$-solution $\bar{\mathcal{T}}$ of (D), as well as outer and inner approximations to the upper image $\mathcal{P}$ (and to the lower image $\mathcal{D}$, by duality).

The feasible region $\mathcal{T}= \{t\in\R^q : w(t) \in C^+\}$ of \eqref{(D)} obviously provides an outer approximation of $\mathcal{D}$. All supporting hyperplanes of $\mathcal{T}$ are {\em vertical}, that is, they have a normal vector $t^*\in\R^q\smz$ with $t^*_q=0$. 

Recall that Assumption~\ref{assmp} was assumed to hold in this section, which leads to the existence of optimal solutions of \eqref{P1} and \eqref{D1} for every $w\in \R^q$. Both convex programs play an important role in the following. 

\begin{proposition}\label{dualprop2}
Let $t \in \mathcal{T}$, $w := w(t)$, and $y^w \in \R$ be the optimal objective value for \eqref{P1}. Then,
$$ t \notin \mathcal{D} \Longleftrightarrow t_q > y_w \qquad \text{and}\qquad t \in \MaxK (\mathcal{D}) \Longleftrightarrow t_q = y_w.$$
\end{proposition} 
\begin{proof}
This follows directly from the definition of $\mathcal{D}$.
\end{proof}

\begin{proposition}\label{dualprop3}
Let $t \in \mathcal{T}$, $w := w(t)$, and $x^w$ be an optimal solution to problem \eqref{P1}. Then $\mathcal{H}^*:=H^*(\Gamma(x^w))$ is a non-vertical supporting hyperplane to $\mathcal{D}$ at $D^*(t)$.
\end{proposition}
\begin{proof}
By Proposition~\ref{dualprop0}, $D^*(t)\in \bd\mathcal{D}$. We have $D^*(t) \in \mathcal{H}^*$, as $w(D^*(t)) = w(t) = w$ and $(D^*(t))_q = w^T\Gamma(x^w)$ imply $\varphi(\Gamma(x^w),D^*(t)) = w(D^*(t))^T\Gamma(x^w)-(D^*(t))_q = 0$.
Since $y^* \mapsto w(y^*)$ is an affine function but does not depend on $y^*_q$, $\mathcal{H}^*$ is a non-vertical hyperplane.
For any $d\in \mathcal{D}$, we have 
$\varphi(\Gamma(x^w),d)=w(d)^T\Gamma(x^w)-d_q\geq \inf_{x\in \mathcal{X}}\left[w(d)^T\Gamma(x)\right] - d_q\geq 0$,
where the last inequality follows from the definition of $\mathcal{D}$.
\end{proof}

The initial outer approximation $\mathcal{D}_0$ is obtained by solving (P$_1(\eta)$) for 
\begin{equation}\label{eta}
	\eta= \frac{1}{J} \sum_{j=1}^J z^j \in \Int C^+,
\end{equation}
where $z^j, j=1,\ldots,J$ are the generating vectors of $C^+$ such that $(z^j)^Tc=1$. 
By Proposition~\ref{dualprop1}, an optimal solution $x^{\eta}$ of (P$_1(\eta)$) is a weak minimizer of \eqref{(P)}.
We have $\eta^Tc=1$ and $w(T^T \eta) = \eta$ by \eqref{eq_wt}. 
By Proposition \ref{dualprop0}, $t:=T^T \eta$ is a maximizer for \eqref{(D)} and $D^*(T^T \eta) = (t_1,\ldots,t_{q-1},y^\eta)$, where $y^{\eta} \in \R$ denotes the optimal value of (P$_1(\eta)$). 
By Proposition~\ref{dualprop3}, $H^*(\Gamma(x^{\eta}))$ is a non-vertical supporting hyperplane to $\mathcal{D}$.
The initial outer approximation is 
\begin{equation}\label{D0}
	\mathcal{D}_0 := \mathcal{T} \cap \{ y^* \in \R^q : \varphi(\Gamma(x^{\eta}),y^*)\geq 0\}\supseteq \mathcal{D}.
\end{equation}
As $\mathcal{D}_0$ contains no lines, it has at least one vertex. We now state the dual algorithm. 

\begin{algorithm}
\caption{Approximation Algorithm: A Dual Variant}
\begin{algorithmic}[1]\label{alg_2}
\STATE Compute an optimal solution $x^\eta$ to (P$_1(\eta)$) for $\eta$ in \eqref{eta};
\STATE Store an H-representation $\mathcal{D}^H$ of  $\mathcal{D}_0$ according to \eqref{D0}; 	
\STATE $k \gets 0$; $\bar{\mathcal{X}} \gets \{x^{\eta}\}$; $\bar{\mathcal{T}} \gets \{T^T \eta\}$; $D \gets D^*(\bar{\mathcal{T}})$, where $D^*_q(T^T\eta) = \eta^T \Gamma(x^{\eta})$;
\REPEAT
	\STATE $M \gets \R^q$;
	\STATE Compute the set $\mathcal{D}^V$ of vertices of $\mathcal{D}_k$ from its H-representation $\mathcal{D}^H$;
		\FOR{$i=1:\abs{\mathcal{D}^V}$}
			\STATE Let $t$ be the $i^{th}$ element of $\mathcal{D}^V$ (i.e. the $i^{th}$ vertex of $\mathcal{D}_k$) and set $w \gets w(t)$;
			\STATE Compute an optimal solution $x^w$ to \eqref{P1} and the optimal value $y^w = w^T\Gamma(x^w)$;		
			\STATE $\bar{\mathcal{X}} \gets \bar{\mathcal{X}} \cup \{x^w\}$;
			\IF{($w \notin \bd{C^+}$ or  $t_q - y^w \leq \epsilon$)}
				\STATE $\bar{\mathcal{T}} \gets \bar{\mathcal{T}} \cup \{t\}$; Update $D$ using $D^*(t)=(t_1,\dots,t_{q-1},y^w)^T$ such that $D=D^*(\bar{\mathcal{T}})$;
			\ENDIF
				\IF{$t_q - y^w > \epsilon$}
					\STATE $M \gets M  \cap \{y^*\in \R^q: \varphi(\Gamma(x^w),y^*)\geq 0\}$; 
					\STATE break; (optional)
				\ENDIF
		\ENDFOR
	\IF{$M\neq\R^q$}
		\STATE Store in $\mathcal{D}^H$ an H-representation of $\mathcal{D}_{k+1} = \mathcal{D}_k \cap M$ and set $k \gets k+1$;
	\ENDIF
\UNTIL  $M=\R^q$
\STATE Compute the vertices $\mathcal{V}$ of $\{y\in\R^q: \varphi(y,y^*)\geq 0, \forall y^* \in D=D^*(\bar{\mathcal{T} } )\}$;
\RETURN $\left\{ \begin{array}{ll}
\bar{\mathcal{X}}&: \text{A finite weak~}\epsilon \text{-solution to~\eqref{(P)};}\\
\bar{\mathcal{T}}&: \text{A finite~}\epsilon \text{-solution to~\eqref{(D)};}\\
\mathcal{V} &:      \text{Vertices of an outer }\epsilon\text{-approximation of }\mathcal{P};\\
\Gamma(\bar{\mathcal{X}}) &: \text{Vertices of an inner }\epsilon\text{-approximation of }\mathcal{P}.\\
\end{array} \right.$
\end{algorithmic}
\end{algorithm}

\begin{theorem}\label{thmdual}
Let Assumption~\ref{assmp} be satisfied. Then, Algorithm~\ref{alg_2} works correctly: If the algorithm terminates, it returns a finite weak $\epsilon$-solution $\bar{\mathcal{X}}$ to problem \eqref{(P)}, and a finite $\epsilon$-solution $\bar{\mathcal{T}}$ to the dual problem \eqref{(D)}.
\end{theorem}
\begin{proof} By Assumption~\ref{assmp}, \eqref{P1} has a solution $x^w$ for every $w \in \R^q$. Note that we have $\mathcal{D}_k \subseteq\mathcal{D}_{k-1} \subseteq\mathcal{D}_0\subseteq\mathcal{T}$ by construction. This ensures that $\eta$ in line 1 and $w$ in line 8 belong to $C^+\smz$. Hence, by Proposition \ref{dualprop1}, $\bar{\mathcal{X}}$ consists of a weak minimizers of \eqref{(P)} only. We know that $\mathcal{D}_k$, $k=0,1,2,...$ contains no lines and, therefore, it has at least one vertex. Every vertex $t$ of $\mathcal{D}_k$ in line 8 belongs to $D^*(\mathcal{T})$. By Proposition~\ref{dualprop0}, $t$ is a maximizer to \eqref{(D)} with $D^*(t)=(t_1,\ldots,t_{q-1},y^w)$. Proposition~\ref{dualprop3} yields that $H^*(\Gamma(x^w))$ is a supporting hyperplane of $\mathcal{D}$ at $D^*(t)$. Hence, we have $\mathcal{D}_k \supseteq \mathcal{D}$ for all $k$.
The condition in line 11 just excludes some of the $t$'s to prevent that multiple elements are added to $\mathcal{T}$ which yield the same objective value $D^*(t)$. 

Assume the algorithm stops after $\hat{k}$ iterations. The vertices $t$ of the outer approximation $\mathcal{D}_{\hat{k}}$ of $\mathcal{D}$ satisfy $t_q -y^w \leq \epsilon$, where $y^w$ is the optimal objective value of \eqref{P1} for $w=w(t)$. 

We next show that $\bar{\mathcal{T}}$ is a finite $\epsilon$-solution to \eqref{(D)}. We know that $\mathcal{T}$ is nonempty and consists of maximizers for \eqref{(D)} only. It remains to show that $\bar{\mathcal{T}}$ is a finite $\epsilon$-supremizer, i.e., \eqref{eqweakdualsoln} holds. As shown above, we have $\mathcal{D}\subseteq \mathcal{D}_{\hat{k}}$. By construction, every vertex of $\mathcal{D}_{\hat{k}}$ belongs to $\bar{\mathcal{T}}$ and we have $t_q-y^{w(t)} \leq \epsilon$. By $D^*(t)=(t_1,\dots,t_{q-1},y^{w(t)})^T$, we obtain $\mathcal{D}_{\hat{k}} \subseteq \conv D^*(\bar{\mathcal{T}}) +\epsilon \{e^q\} - K =: \mathcal{D}_{\epsilon}$.

Finally, we prove that $\bar{\mathcal{X}}$ is a finite weak $\epsilon$-solution to \eqref{(P)}. We already know that $\bar{\mathcal{X}}$ is nonempty and finite, and it consists of weak minimizers for \eqref{(P)} only. It remains to show that $\bar{\mathcal{X}}$ is a finite $\epsilon$-infimizer for \eqref{(P)}, i.e., \eqref{eqweaksoln} holds. Setting $\bar{\mathcal{P}}:=\conv \Gamma(\bar{\mathcal{X}})+C$, we have
$$\bar{\mathcal{D}}:=\{y^*\in\R^q: \forall y\in \bar{\mathcal{P}},\, \varphi(y,y^*)\geq 0\} \subseteq\{y^*\in\R^q: \forall y \in \Gamma(\bar{\mathcal{X} }),\, \varphi(y,y^*)\geq 0\}\subseteq \mathcal{D}_{\hat{k}} \subseteq \mathcal{D}_{\epsilon}.$$
Using Proposition~\ref{apprprop1}, we conclude 
$$\bar{\mathcal{P}} =  \{y\in \R^q: \forall y^* \in \bar{\mathcal{D}},\, \varphi(y,y^*) \geq 0\} \supseteq \{y\in \R^q: \forall y^* \in \mathcal{D}_{\epsilon},\, \varphi(y,y^*) \geq 0\} =: \mathcal{P}_{\epsilon}.$$
By Proposition~\ref{approximationremark}, ${\mathcal{P}}_\epsilon + \epsilon \{c\} \supseteq \mathcal{P}$. Altogether we have $\conv \Gamma(\bar{\mathcal{X}})+C + \epsilon \{c\} \supseteq \mathcal{P}$.
\end{proof}

\begin{remark}[Alternative Algorithm~\ref{alg_2}]\label{alternativedual}
Using similar arguments as in Remark~\ref{alternative} for Algorithm~\ref{alg_1}, we obtain an alternative variant of Algorithm~\ref{alg_2}. It is possible to replace lines~10-17 of Algorithm~\ref{alg_2} by the alternative lines given below. In addition, one can initialize $\bar{\mathcal{T}}$ as the empty set in line~3. Line 23 can be skipped as the vertices of a coarser outer $\epsilon$-approximation of $\mathcal{P}$ are also given by $\mathcal{D}^V$. 
\end{remark}
\begin{algorithm}
\caption*{Alternative to Lines 10-17 of Algorithm~\ref{alg_2}}
\begin{algorithmic}
\IF{$t_q - y^w \leq \epsilon$}
	\STATE $\bar{\mathcal{T}} \gets \bar{\mathcal{T}} \cup \{t\}$;
	\STATE Update $D$ using $D^*(t)=(t_1,\dots,t_{q-1},y^w)^T$ such that $D=D^*(\bar{\mathcal{T}})$;
\ELSE
	\STATE $\bar{\mathcal{X}} \gets \bar{\mathcal{X}} \cup \{x^w\}$;
	\STATE $M \gets M  \cap \{y^*\in \R^q: \varphi(\Gamma(x^w),y^*)\geq 0\}$; 
	\STATE break; (optional) 
\ENDIF
\end{algorithmic}
\end{algorithm}

\subsection{Remarks}
\label{rems}
1. Algorithms~\ref{alg_1} and \ref{alg_2} provide finite $\epsilon$-solutions to \eqref{(D)}, but only a finite `weak' $\epsilon$-solution to \eqref{(P)}. The reason can be found in Propositions~\ref{dualprop1}, and~\ref{primalprop2}. Recall that for LVOPs in case of $\epsilon=0$ the situation is different: One can easily find a finite solution to \eqref{(P)}. Since $\mathcal{P}$ is polyhedral, the vertices that generate $\mathcal{P}$ are considered, and any vertex of $\mathcal{P}$ is $C$-minimal, see \cite{lvop}.

2. Algorithms~\ref{alg_1} and \ref{alg_2} return vertices of inner and outer approximations of $\mathcal{P}$. However,
both inner and outer approximations of $\mathcal P$ and $\mathcal D$ can be obtained from the  $\epsilon$-solution concept: If $\bar{\mathcal{X}}$ is a finite weak $\epsilon$-solution to \eqref{(P)} and $Y$ is the matrix of generating vectors of $C$, then $\conv \Gamma(\bar{\mathcal{X}}) + \cone \{ y \in \R^q: y \text{ column of } Y\}$ is a V-representation of an inner approximation of $\mathcal P$ and $\{y^*\in\R^q : \forall y \in \Gamma(\bar{\mathcal{X}}),\, \varphi(y,y^*)\geq 0,\;  Y^T w(y^*) \geq 0\}$ is an H-representation of an outer approximation of $\mathcal D$. If $\bar{\mathcal{T}}$ is a finite $\epsilon$-solution to \eqref{(D)}, then $\conv D^*(\bar{\mathcal{T}}) + \cone\{-e^q\}$ is a V-representation of an inner approximation of $\mathcal{D}$, and $\{y\in\R^q: \forall y^* \in D^*(\bar{\mathcal{T}}),\, \varphi(y,y^*) \geq 0)\}$ is an H-representation of an outer approximation of $\mathcal P$.

3. All algorithms in this paper still work correctly if Assumption~\ref{assmp} (a) is replaced by the requirement that optimal solutions of \eqref{P1} and \eqref{P2} exist for those parameter vectors $w$ and $v$ that occur during an algorithm is applied to a problem instance. This follows from the fact that compactness was only used to prove existence of the scalar problems. Note that \eqref{(P)} being bounded was not explicitly used in any proof. However, \eqref{(P)} being bounded is equivalent to \eqref{P1} being bounded for every $w\in C^+$, which is necessary but not sufficient for the existence of a solution to \eqref{P1}. Problems with non-compact feasible set are solved in Examples~\ref{ex3} and~\ref{caginex1} below. 

4. This last remark concerns finiteness of the algorithms presented here. Even though there exists a finite $\epsilon$-solution to \eqref{(P)} for any error level $\epsilon >0$ by Proposition~\ref{existence_soln}, and a finite weak $\epsilon$-solution to \eqref{(P)} is found if the algorithm terminates, it is still an open problem to show that the algorithms are finite. If the optional break command in Algorithm~\ref{alg_1} is disabled, every vertex of the current approximation of $\mathcal{P}$ is checked before updating the next approximation. Thus, $\epsilon_k := \max\{z^v: v \text{ vertex of } \mathcal{P}_k\} \geq 0$ decreases in each iteration of Algorithm~\ref{alg_1}. If one stops the algorithms at the $\hat{k}^{th}$ iteration, a finite weak $\epsilon_{\hat{k}}$- solution to \eqref{(P)} is returned. 
In order to show that the algorithm is finite, one would need to show that $\lim_{k \rightarrow \infty} \epsilon_{k} = 0$. The situation in Algorithm~\ref{alg_2} is similar.
 
\section{Examples and numerical results}
\label{secex}

We provide four examples in this section. The first one illustrates how the algorithms work. The second example is \cite[Example 6.2]{ehrgott}, an example with non-differentiable constraint function. This problem can be solved by the algorithms in this paper, where we use a solver which is able to solve some special non-differential scalar problems. The third example has three objectives and is related to generalization (iii) in the introduction. The last example has a four dimensional outcome space and shows that the algorithms provided here can be used for the calculation of set-valued convex risk measures. The corresponding vector optimization problems naturally have ordering cones being strictly larger and having more generating vectors than $\R^q_+$, which was one motivation to extend the algorithms to arbitrary solid convex polyhedral ordering cones.

We provide some computational data for each example with a corresponding table. The second column of the table shows the variant of the algorithm, where `break/no break' corresponds to the optional breaks in line 14 of Algorithm~\ref{alg_1}, and line 16 of Algorithm~\ref{alg_2}. The next two columns show the number of scalar optimization problems ($\#$ opt.) solved, and the number of vertex enumerations ($\#$ vert. enum.) used during the algorithms. We provide the number of elements in the solution sets $\abs{\bar{\mathcal{X}}}$, $\abs{\bar{\mathcal{T}}}$ found by Algorithms~\ref{alg_1} and~\ref{alg_2}, as well as the number of the elements of the solution sets $\abs{\bar{\mathcal{X}}_{alt}}$, $\abs{\bar{\mathcal{T}}_{alt}}$ found by the alternative versions given by Remarks~\ref{alternative} and~\ref{alternativedual}. CPU time is measured in seconds. We used MATLAB to implement the algorithms, and we employ CVX, a package for specifying and solving convex programs, as a solver (\cite{cvx, boyd}). We ignore the approximation error of the solver, because it is typically much smaller than the $\epsilon$ we fix in the algorithms.

\begin{example}
\label{ex1}
Consider the following problem
\begin{align*}
&\text{minimize~~} \Gamma(x) = (x_1,x_2)^T \text{~~with respect to ~}\leq_{\R^2_+}\\
&\text{subject to~~} (x_1-1)^2 + (x_2-1)^2 \leq 1,\;\; x_1, x_2 \geq 0.
\end{align*}
We set $c^1 = [0,1]^T, c = [1,1]^T$. The corresponding upper image $\mathcal{P}$ and the lower image $\mathcal{D}$ can be seen in Figure~\ref{figure0}.

\begin{figure}[ht]
\centering
\includegraphics[width=15cm,height=5cm]{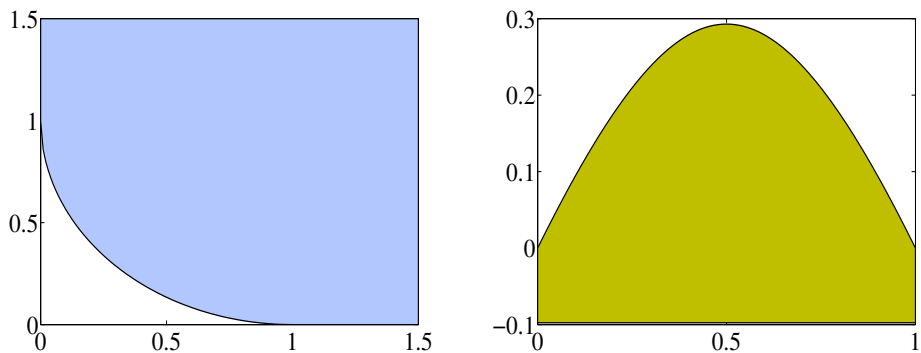}
\caption{Upper (left) and lower (right) images for Example~\ref{ex1}.}
\label{figure0}
\end{figure}

Algorithm~\ref{alg_1} starts with $\mathcal{P}_0 = C=\R^2_+$. We set the approximation error $\epsilon = 0.05$ and use the optional break in line 14. Figure~\ref{figure1} shows the inner and outer approximations for the upper image after the first three iterations $k=1,2,3$. Remember that the current outer approximation is $\mathcal{P}_k$. The current inner approximation is $\conv \Gamma(\bar{\mathcal{X}}_k)+C$, where $\bar{\mathcal{X}}_k$ denotes the `solution' set $\bar{\mathcal{X}}$ at the end of iteration step $k$.

\begin{figure}[ht]
\hspace{-1.5cm}\includegraphics[width=18cm,height=4.5cm]{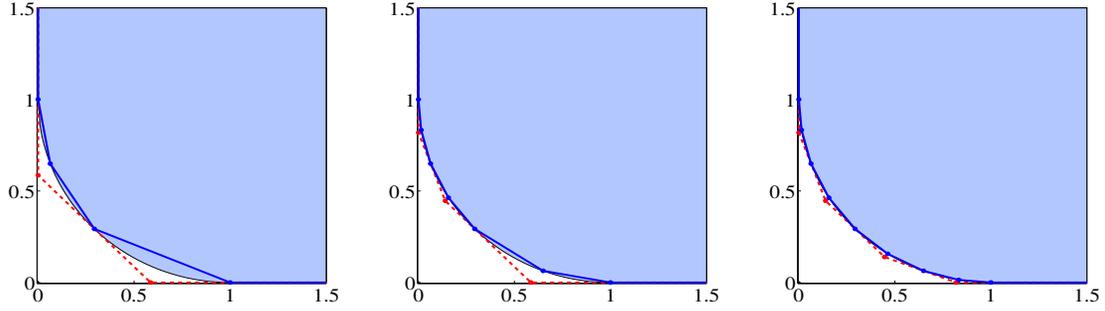}
\caption{The first three iterations of Algorithm~\ref{alg_1} with `break' in line 14.}
\label{figure1}
\end{figure}

After the third iteration ($\hat k=3$), Algorithm~\ref{alg_1} stops. It returns $9$ vertices of an inner approximation given by $\Gamma(\bar{\mathcal{X}})$, which coincide in this example with the 
finite weak $\epsilon$-solution $\bar{\mathcal{X}} $ to \eqref{(P)}: 
\begin{align*}
\bar{\mathcal{X}} 
&= \bigg\{ \left[\begin{array}{ccc}
0 \\
1 \end{array}\right], \left[\begin{array}{ccc}
0.0141 \\
0.8329 \end{array}\right], \left[\begin{array}{ccc}
0.0635 \\
0.6493 \end{array}\right], \left[\begin{array}{ccc}
0.1564 \\
0.4631 \end{array}\right], \left[\begin{array}{ccc}
0.2929 \\
0.2929 \end{array}\right], \\
& \:\:\:\:\:\:\:\:\:\:\:\:\: \left[\begin{array}{ccc}
0.4631  \\
0.1564 \end{array}\right], \left[\begin{array}{ccc}
0.6493 \\
0.0635 \end{array}\right], \left[\begin{array}{ccc}
0.8329 \\
0.0141 \end{array}\right], \left[\begin{array}{ccc}
1\\
0 \end{array}\right]\bigg\}.
\end{align*} 

The $8$ vertices $\mathcal{V}$ of the final outer approximation are calculated in line~$21$ of Algorithm~\ref{alg_1}, see left picture in Figure~\ref{finalfigure1}. If one uses the alternative version of the algorithm explained by Remark~\ref{alternative}, then the inner and outer approximation would only have four vertices, see right picture in Figure~\ref{finalfigure1}, and a
finite weak $\epsilon$-solution $\bar{\mathcal{X}} $ to \eqref{(P)} is calculated as
\begin{align*}
\bar{\mathcal{X}} 
&= \bigg\{ \left[\begin{array}{ccc}
0.0141 \\
0.8329 \end{array}\right], \left[\begin{array}{ccc}
0.1564 \\
0.4631 \end{array}\right], \left[\begin{array}{ccc}
0.4631 \\
0.1564 \end{array}\right], \left[\begin{array}{ccc}
0.8329 \\
0.0141 \end{array}\right] \bigg\}.
\end{align*}

\begin{figure}[ht]
\centering
\includegraphics[width=14cm,height=5cm]{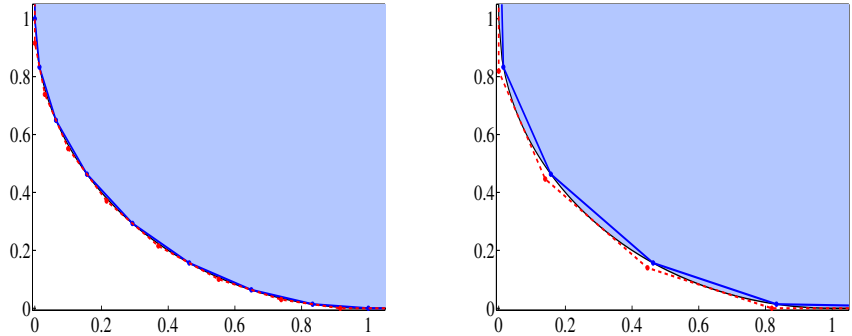}
\caption{The inner and outer approximations of the upper image provided by Algorithm \ref{alg_1} (left) and its alternative (right).}
\label{finalfigure1}
\end{figure}

Now, let us solve the same example using the dual variant of the algorithm. We use the same approximation error $\epsilon = 0.05$ and use the optional break in line 16. Note that $\mathcal
T  = \{ t=(t_1,t_2)\in \R^2 : 0 \leq t_1 \leq 1\}$. The initial non-vertical supporting hyperplane of $\mathcal{D}$ is found as $\{ t \in \R^2 : t_2 = 0.2929 \}$. The vertices of the initial outer approximation are then $(0,\; 0.2929)^T$ and $(1,\; 0.2929)^T$. Note that the current outer approximation is $\mathcal{D}_k$, while the current inner approximation is $\conv D^*(\bar{\mathcal{T}}_k)-K$, where $\bar{\mathcal{T}}_k$ denotes the `solution' set $\bar{\mathcal{T}}$ at the end of iteration step $k$. Figure~\ref{figure2} shows the approximations of the lower image after the first four iterations ($k=1,2,3,4$). The computational data can be seen in Table~\ref{table1}.

\begin{figure}[ht]
\hspace{-1.5cm}
\includegraphics[width=18.5cm,height=4.4cm]{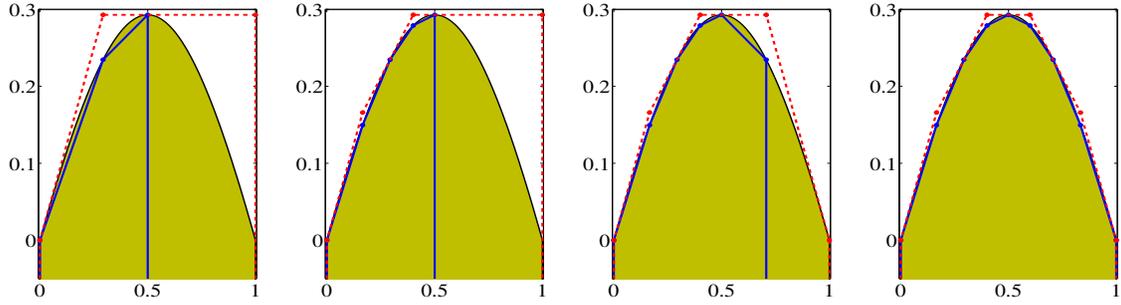}
\caption{The first four iterations of Algorithm~\ref{alg_2} with `break' in line 16.}
\label{figure2}
\end{figure}

After the fourth iteration the algorithm stops ($\hat k=4$). The algorithm calculated the $9$ vertices $D^*(\bar{\mathcal{T}})$ of the inner approximation of $\mathcal D$ (see right picture in Figure~\ref{figure2} or left picture in  Figure~\ref{finalfigure2}). An H-representation of the final outer approximation  of $\mathcal D$ is given by $\{y^*\in\R^q : \forall y \in \Gamma(\bar{\mathcal{X}}),\, \varphi(y,y^*)\geq 0,\;  Y^T w(y^*) \geq 0\}$, see Section~\ref{rems}. Its $10$ vertices can be calculated by vertex enumeration, see left picture in Figure~\ref{finalfigure2}.
The algorithm returns a finite weak $\epsilon$-solution $\bar{\mathcal{X}} $ to \eqref{(P)} as follows:
\vspace{-0.1cm}
\begin{align*}
\bar{\mathcal{X}} &= \bigg\{ \left[\begin{array}{ccc}
0 \\
1      \end{array}\right], \left[\begin{array}{ccc}
0.0192 \\
0.8049 \end{array}\right], \left[\begin{array}{ccc}
0.0761 \\
0.6173 \end{array}\right], \left[\begin{array}{ccc}
0.1685 \\
0.4445 \end{array}\right], \left[\begin{array}{ccc}
0.2929 \\
0.2929      \end{array}\right],  \\
& \:\:\:\:\:\:\:\:\:\:\:\:\:\:\:\:\:\:\:\:\:\:\:\:\:\:\:\:\:\:\:\:\: \left[\begin{array}{ccc}
0.4445  \\
0.1685 \end{array}\right], \left[\begin{array}{ccc}
0.6173 \\
0.0761 \end{array}\right], \left[\begin{array}{ccc}
0.8049 \\
0.0192 \end{array}\right], \left[\begin{array}{ccc}
1\\
0 \end{array}\right] \bigg\}. 
\end{align*}

\vspace{-0.1cm}\noindent
If one uses the alternative version of the algorithm explained by Remark~\ref{alternativedual}, then a finite weak $\epsilon$-solution $\bar{\mathcal{X}} $ to \eqref{(P)} is found as
\vspace{-0.1cm}
\begin{align*}
\bar{\mathcal{X}} = \bigg\{ \left[\begin{array}{ccc}
0 \\
1 \end{array}\right], \left[\begin{array}{ccc}
0.0761 \\
0.6173 \end{array}\right], \left[\begin{array}{ccc}
0.2929 \\
0.2929 \end{array}\right], \left[\begin{array}{ccc}
0.6173 \\
0.0761 \end{array}\right], \left[\begin{array}{ccc}
1 \\
0 \end{array}\right]\bigg\},
\end{align*}

\vspace{-0.1cm}\noindent
and the final inner and outer approximations of $\mathcal D$ are given as in the right picture of Figure~\ref{finalfigure2}.
\begin{figure}[ht]
\centering
\includegraphics[width=15cm,height=5.2cm]{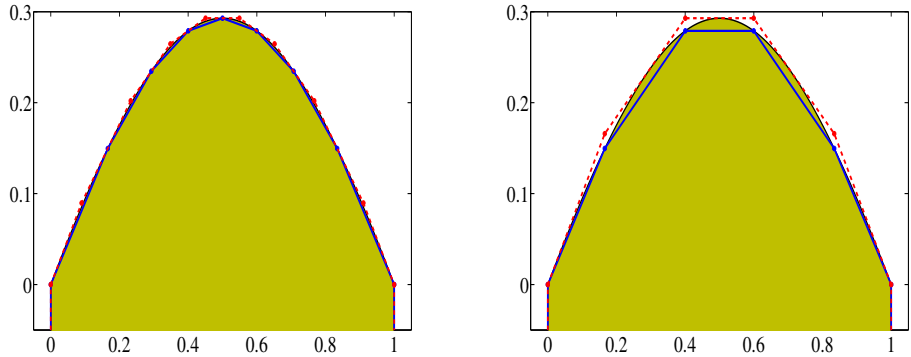}
\caption{The inner and outer approximations of the lower image $\mathcal D$  provided by Algorithm~\ref{alg_2} (left) and its alternative (right).}
\label{finalfigure2}
\end{figure}
 
\begin{table}[ht]
\caption{Computational data for Example~\ref{ex1}}
\centering
\begin{tabular}{ |l|l|l|l|l|l|l|l|l| }
\hline\hline
& & & & & & & & \\ [-2.3ex]
$\epsilon$ & alg. / variant & $\#$ opt. & $\#$ vert. enum. & $\abs{\bar{\mathcal{X}}}$ & $\abs{\bar{\mathcal{T}}}$ & $\abs{\bar{\mathcal{X}}_{alt}}$ & $\abs{\bar{\mathcal{T}}_{alt}}$ & time (s)\\ [0.5ex]
\hline\hline
\multirow{4}{*}{$0.01$}   	& \ref{alg_1} / break    & $17$  & $9$  & $17$ & $17$ & $8$  & $9$   & $8.37$ \\ 
							& \ref{alg_1} / no break & $17$  & $5$  & $17$ & $17$ & $8$  & $9$   & $8.26$ \\
							& \ref{alg_2} / break    & $19$  & $11$ & $17$ & $17$ & $8$  & $10$  & $8.92$ \\
							& \ref{alg_2} / no break & $19$  & $6$  & $17$ & $17$ & $8$  & $10$  & $8.85$ \\ \hline
\multirow{4}{*}{$0.001$}	& \ref{alg_1} / break    & $45$  & $23$ & $45$  & $45$  & $22$  & $23$  & $21.18$ \\
							& \ref{alg_1} / no break & $45$  & $7$  & $45$  & $45$  & $22$  & $23$  & $21.40$ \\
							& \ref{alg_2} / break    & $43$  & $23$ & $41$  & $41$  & $20$  & $22$  & $20.24$ \\
							& \ref{alg_2} / no break & $43$  & $8$  & $41$  & $41$  & $20$  & $22$  & $20.00$ \\ \hline
\end{tabular}
\label{table1}
\end{table}
\end{example}

\begin{example}
\label{ex2}
Consider the following problem with non-differentiable constraint function
\vspace{-0.1cm}
\begin{align*}
&\text{minimize~~} \Gamma(x) = \binom{(x_1-3)^2+(x_2-1)^2}{(x_1-1)^2 + (x_2-1)^2} \text{~with respect to~}\leq_{\R^2_+}\\
&\text{subject to~} |x_1| + 2|x_2| \leq 2.
\end{align*}

\vspace{-0.1cm}\noindent
The ordering cone is $C=\R^2_+$, and we fix $c^1 = [1,0]^T, c=[1,1]^T$ as before. This example is taken from \cite{ehrgott}, and it was used as an example which can not be solved by the algorithm provided in \cite{ehrgott}. Since we do not assume differentiability in order to use the algorithms provided here, the example can be solved. Figure~\ref{nondif_primal} shows the approximations of the upper and lower images generated by Algorithm~\ref{alg_1}, where the approximation error $\epsilon$ is taken as $0.01$. Computational data regarding this example can be seen in Table~\ref{table2}.

\begin{figure}[ht]
\centering
\includegraphics[width=15cm,height=5cm]{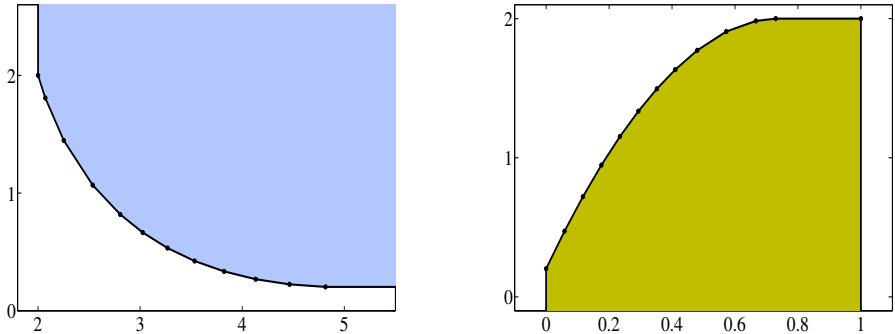}
\caption {The outer approximation of the upper image (left) and the inner approximation of the lower image (right) by Algorithm~\ref{alg_1} for $\epsilon = 0.01$ for Example~\ref{ex2}.}
\label{nondif_primal}
\end{figure}

\begin{table}[ht]
\caption{Computational data for Example~\ref{ex2}}
\centering
\begin{tabular}{ |l|l|l|l|l|l|l|l|l| }
\hline\hline
& & & & & & & & \\ [-2.3ex]
$\epsilon$ & alg. / variant & $\#$ opt. & $\#$ vert. enum. & $\abs{\bar{\mathcal{X}}}$ & $\abs{\bar{\mathcal{T}}}$ & $\abs{\bar{\mathcal{X}}_{alt}}$ & $\abs{\bar{\mathcal{T}}_{alt}}$ & time (s)\\ [0.5ex]
\hline\hline
\multirow{4}{*}{$0.01$}   	& \ref{alg_1} / break    & $25$  & $13$ & $24$ & $25$ & $12$  & $13$  & $16.52$ \\ 
							& \ref{alg_1} / no break & $25$  & $6$  & $24$ & $25$ & $12$  & $13$  & $15.18$ \\
							& \ref{alg_2} / break    & $27$  & $15$ & $25$ & $25$ & $12$  & $14$  & $17.59$ \\
							& \ref{alg_2} / no break & $27$  & $7$  & $25$ & $25$ & $12$  & $14$  & $17.39$ \\ \hline
\multirow{4}{*}{$0.001$}	& \ref{alg_1} / break    & $55$  & $28$ & $54$  & $55$ & $27$  & $28$  & $33.29$ \\
							& \ref{alg_1} / no break & $55$  & $7$  & $54$  & $55$ & $27$  & $28$  & $33.31$ \\
							& \ref{alg_2} / break    & $51$  & $27$ & $49$  & $49$ & $24$  & $26$  & $32.12$ \\
							& \ref{alg_2} / no break & $51$  & $8$  & $49$  & $49$ & $24$  & $26$  & $31.40$ \\ \hline
\end{tabular}
\label{table2}
\end{table}
\end{example}

\vspace{-0.6cm}
\begin{example}
\label{ex3}
Consider the following problem
\vspace{-0.2cm}
\begin{align*}
&\text{minimize~~} \Gamma(x) = (e^{x_1}+e^{x_4},\;e^{x_2}+e^{x_5},\;e^{x_3}+e^{x_6})^T \text{~with respect to }\leq_{\R^3_+}\\
& \text{subject to~~} x_1 + x_2 + x_3 \geq 0\\
&\hspace{1.95cm} 3x_1 + 6x_2 + 3x_3 + 4x_4 + x_5 + 4x_6 \geq 0\\
&\hspace{1.95cm} 3x_1 + x_2 + x_3 + 2x_4 + 4x_5 + 4x_6 \geq 0.
\end{align*}

\vspace{-0.2cm}\noindent
We fix $c^1=[1,0,0]^T$, $c^2=[0,1,0]^T$, and $c=[1,1,1]^T$. Note that Assumptions~\ref{assmp}~(b)-(d) hold, however the feasible region is not compact. Recall that one can still use the algorithms as long as the scalar problems have optimal solutions and in case the algorithms terminate, compare Remarks~3. and~4. in Section~\ref{rems}. We employ the convex optimization solver CVX, which detects whether the scalar problems are infeasible or unbounded, but not necessarily the case where a solution does not exist. If we apply Algorithm~\ref{alg_2} to this example, for the unit vectors $w=e^i$ ($i=1,2,3$), the solver does not detect that \eqref{P1} does not have a solution and returns an approximate solution $x$, where $\norm{\Gamma(x)}$ is very large. By numerical inaccuracy, we cannot ensure that the hyperplane $H^*(\Gamma(x))$ according to Proposition \ref{dualprop3} is non-vertical. This is the reason why the dual algorithm does not work for this example. However we can still use Algorithm~\ref{alg_1}. Figure~\ref{ex3_dual} shows the outer approximations to the upper image generated by Algorithm~\ref{alg_1} with approximation errors $0.1$ and $0.05$. The graphics have been generated by JavaView\footnote{by Konrad Polthier, http://www.javaview.de}. The numerical results can be seen in Table~\ref{table3}. 
This is an example where the algorithm proposed in \cite{ehrgott} does not terminate for certain $\epsilon>0$ and $p \in \Int \mathcal{P}$ (see \cite{ehrgott}) due to a different measure for the approximation error.
\begin{figure}[ht]
\centering
\includegraphics[width=7.7cm,height=5cm]{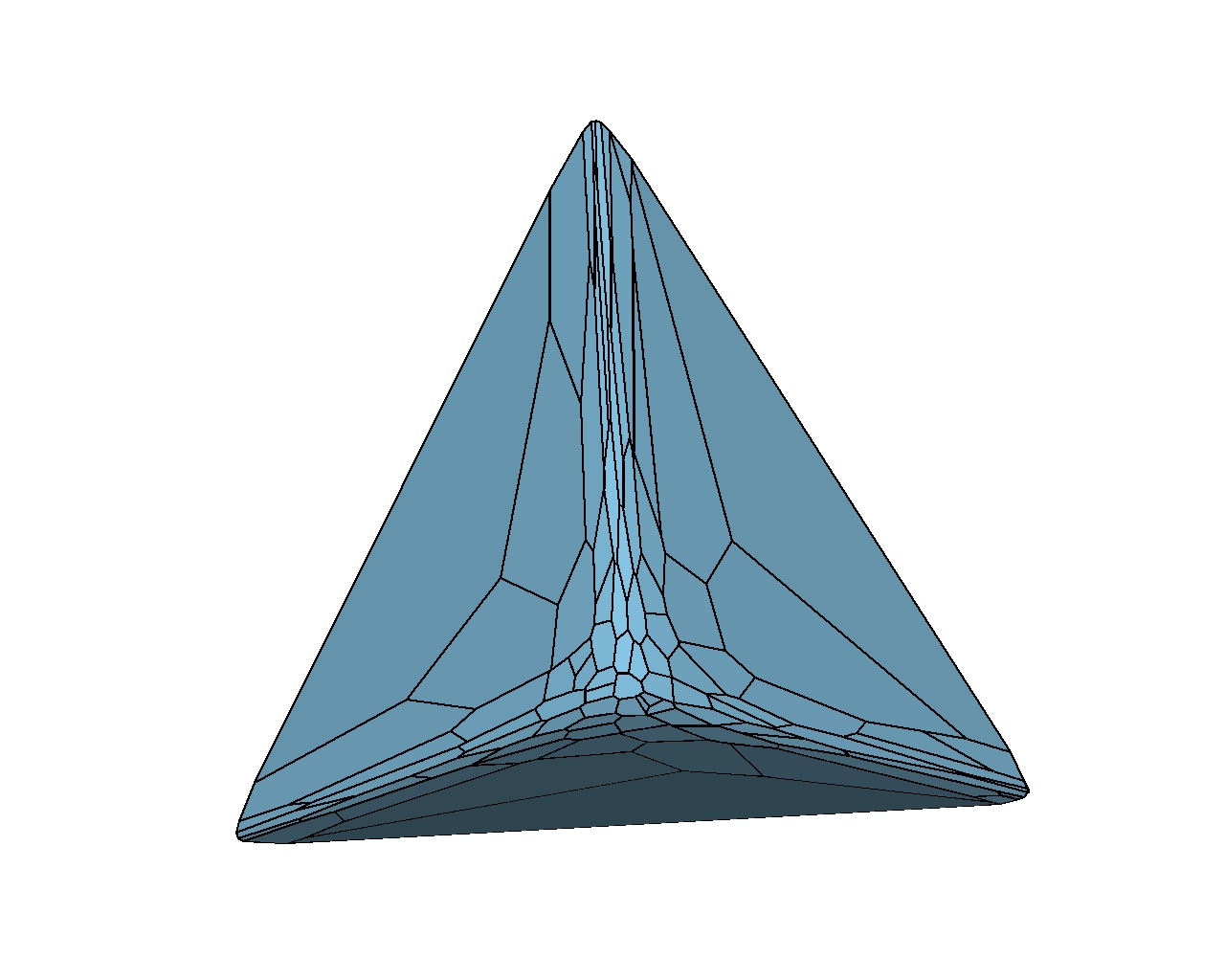}
\includegraphics[width=7.7cm,height=5cm]{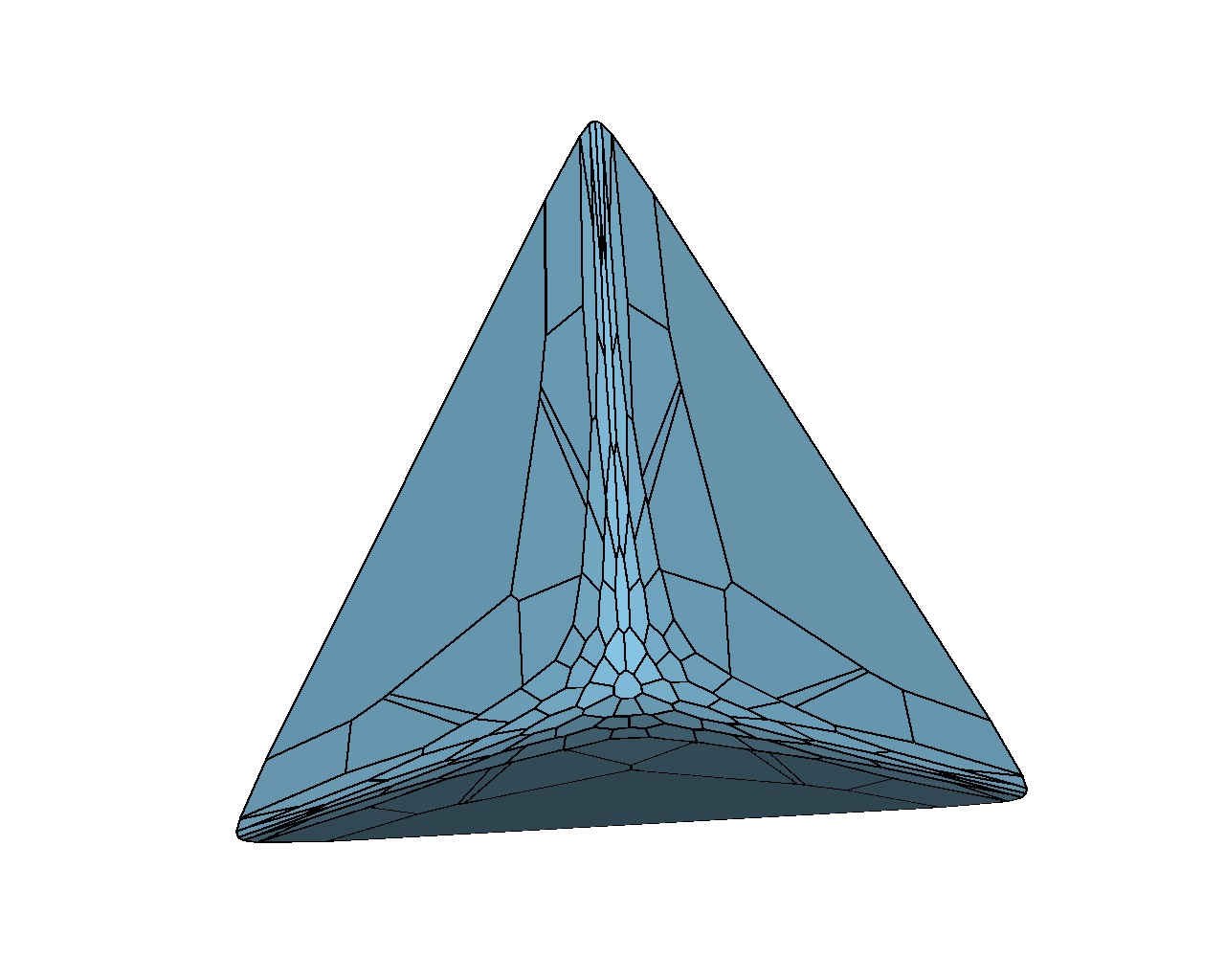}\\
\vspace{-0.5cm}
\includegraphics[width=7.7cm,height=5cm]{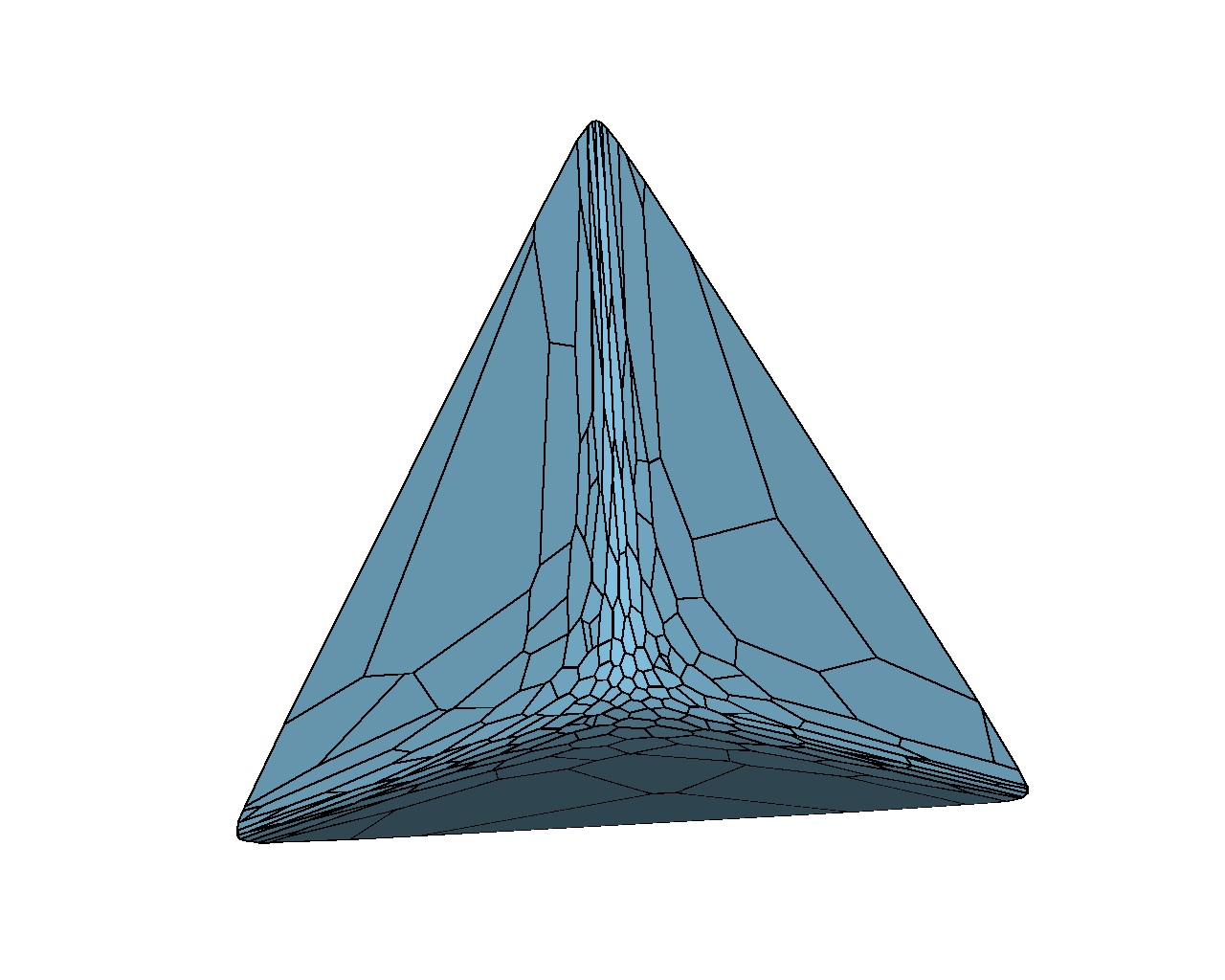}
\includegraphics[width=7.7cm,height=5cm]{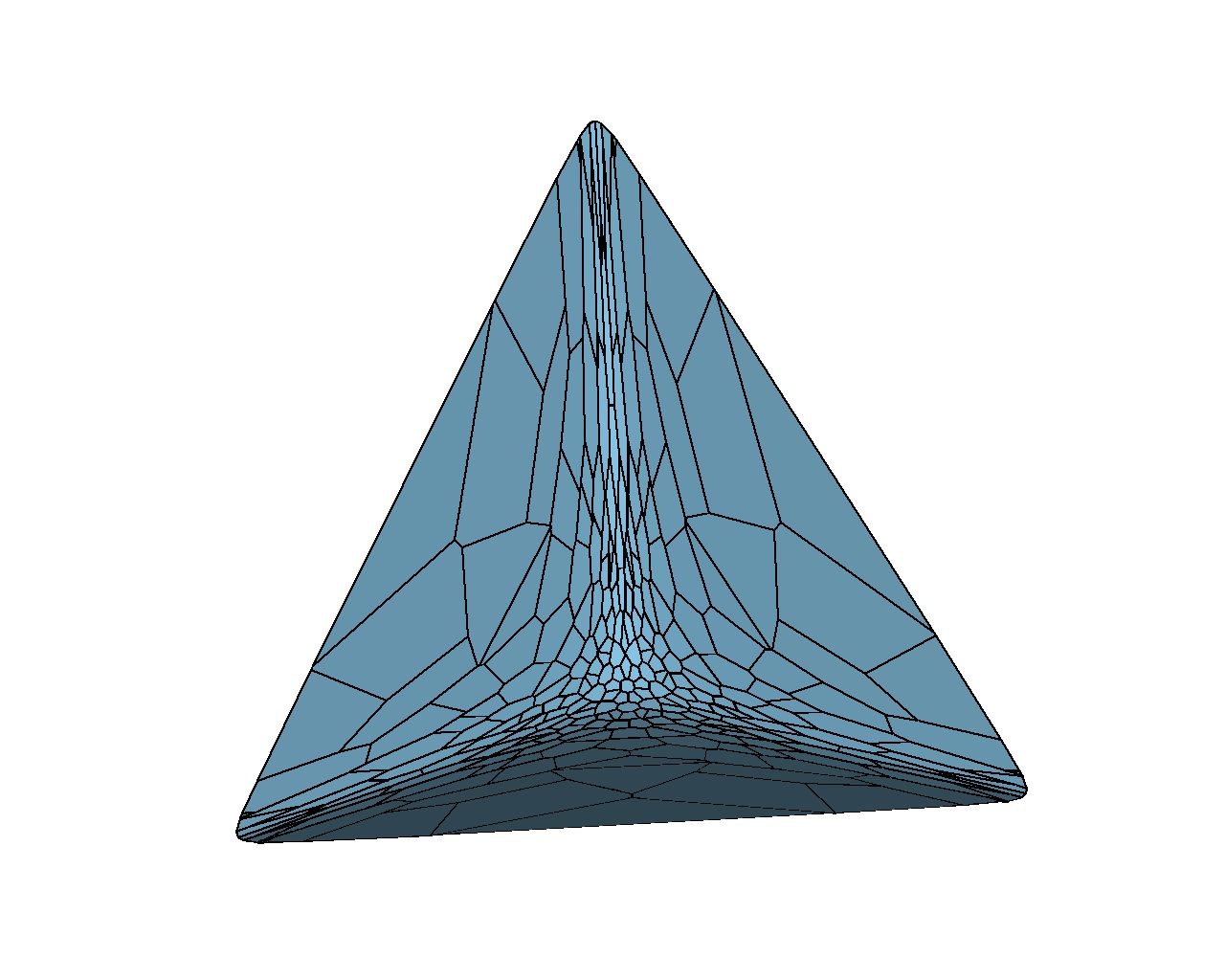}

\caption{The outer approximation of the upper image $\mathcal{P}$ for Example~\ref{ex3} by Algorithm~\ref{alg_1}; displayed: $y\in \mathcal{P}$ with $y_1+y_2+y_3\leq 36$; top left: `break', $\epsilon = 0.1$; top right: `no break', $\epsilon = 0.1$; bottom left: `break', $\epsilon = 0.05$; bottom right: `no break', $\epsilon = 0.05$.}
\label{ex3_dual}
\end{figure}
\end{example}

\begin{table}[ht]
\caption{Computational data for Example~\ref{ex3}}
\centering
\begin{tabular}{ |l|l|l|l|l|l|l|l|l| }
\hline\hline
& & & & & & & & \\ [-2.3ex]
$\epsilon$ & alg. / variant & $\#$ opt. & $\#$ vert. enum. & $\abs{\bar{\mathcal{X}}}$ & $\abs{\bar{\mathcal{T}}}$ & $\abs{\bar{\mathcal{X}}_{alt}}$ & $\abs{\bar{\mathcal{T}}_{alt}}$ & time (s)\\ [0.5ex]
\hline\hline
\multirow{2}{*}{$0.1$}   	& \ref{alg_1} / break    & $107$  & $34$ & $107$  & $107$  & $72$ & $35$  & $284.68$\\ 
							& \ref{alg_1} / no break & $133$  & $7$  & $133$  & $133$  & $87$ & $46$  & $355.29$\\
							  \hline
\multirow{2}{*}{$0.05$}		& \ref{alg_1} / break    & $232$ & $72$ & $231$ & $232$ & $158$ & $73$ & $670.55$\\
							& \ref{alg_1} / no break & $328$ & $8$  & $328$ & $328$ & $222$ & $106$ & $906.91$\\
							  \hline
\end{tabular}
\label{table3}
\end{table}

\begin{example}%
\label{caginex1}
In this example we study the calculation of  set-valued convex risk measures. It is known that polyhedral set-valued convex risk measures can be calculated by Benson's algorithm for LVOPs (see \cite{superhedging,avar,lvop}). Here, we show that (non-polyhedral) set-valued convex risk measures can be calculated approximately using the algorithms provided here. 
Consider a financial market consisting of $d$ assets, which can be traded at discrete time $t = 0,1,\ldots,T$. Let $(\Omega, \mathcal{F},(\mathcal{F})_{t=0}^T,\P)$ be a finite probability space, where $\Omega=\{\omega_1,\ldots,\omega_N\}$, $\P(\omega_i) = p_i$ such that $p_i \in (0,1]$, and $\mathcal{F}=\mathcal{F}_T$. Assume that the market is defined by the $\mathcal{F}_t$ adapted process $(K_t)_{t=0}^T$ of solvency cones, which represent the exchange rates and proportional transaction costs between the $d$ assets (see \cite{kabanov}). Note that $K_t(\omega_n)$ is a polyhedral closed convex cone, with $\R^d_+ \subseteq K_t(\omega_n) \neq \R^d$ for $t = 0,\ldots,T$, $n = 1,\ldots, N$.

In this setting Ararat, Rudloff, and Hamel \cite{cagin} consider set-valued shortfall risk measures on $L^{\infty}_d=L^{\infty}_d(\Omega, \mathcal{F}_T,\P)$, the linear space of the equivalence classes of $\mathcal{F}_T$-measurable, $\P$-a.s. bounded, $d$-dimensional random vectors. Let $l = (l_1,\ldots,l_d):L^{\infty}_d\rightarrow \R^d$ be a loss function such that $l_i:L^{\infty}_d\rightarrow \R$ is convex and increasing for $i=1,\ldots,d$. We consider here the one-period case $T=1$ only. It was shown that the market extension of the set-valued
 $l$-shortfall risk measure is
\vspace{-0.1cm}
\begin{align*}
R^{mar}(X) &= \{m\in\R^d: \: \E\left[l(-X+Y-m+y)\right] \in \{x_0\}-A,\:\: Y \in L^{\infty}_d(K_1),\:\:y\in K_0 \},
\end{align*}

\vspace{-0.1cm}\noindent
where $X \in L^{\infty}_d$, $L^{\infty}_d(K_1) := \{Z \in L^{\infty}_d: \; Z \in K_1 \;\P\text{-a.s.}\}$, $A$ is a convex upper closed set, that is $\cl (\conv (A+\R^d_+)) = A $, and $0\in \bd A$. It is shown in \cite{cagin} that $R^{mar}(X) = R^{mar}(X) + K_0$. For a given random vector $X$, the set $R^{mar}(X)$ is the upper image of the following vector optimization problem
\vspace{-0.1cm}
\begin{align*}
&\text{minimize~~} m \in \mathbb{R}^d \text{~~~with respect to~} \leq_{K_0} \\
&\text{subject to~~} \E\left[l(-X+Y-m+y)\right] \in \{x_0\}-A,\\
&\qquad\qquad\quad Y \in L^{\infty}_d(K_1),\; y \in K_0,
\end{align*}

\vspace{-0.1cm}\noindent
and thus can be approximated by the algorithms presented in this paper, whenever the set $A$ is polyhedral, say $A=\{z\in\R^d:(a^1)^Tz\geq b_1,\ldots,(a^s)^Tz\geq b_s\}$. Then, the constraint $\E\left[l(-X+Y-m+y)\right] \in \{x_0\}-A$ can be written as $$g_i(Y,y,m):= (a^i)^T\big[-x_0+\E\left[l(-X+Y-m+y)\right]\big] \leq -b_i, \:\:\: i=1,\ldots,s.$$ As $A$ is upper closed we have $a^i \in\R^d_+$, which implies that $g_i$ is convex. It is clear that the second and third set of constraints, and the objective function are linear. Thus, the problem is a CVOP with ordering cone $K_0$. Similar to Example~\ref{ex3}, Assumptions~\ref{assmp}~(b)-(d) hold, however the feasible region is not compact as $m\in\R^d$ is not bounded. Remarks~3. and~4. in Section~\ref{rems} explain the implications of that for the algorithms. 

For a numerical example, set $d=4$, $T=1$, $|\Omega| = 8$, and $p_n = 0.125$ for $n=1,\ldots,8$. We fix $\hat K_0 \in \R^{4\times 12}$, and $\hat K_T(\omega_n) \in \R^{4\times 12}$, whose columns are the generating vectors of corresponding solvency cones $K_0$ and $K_T(\omega_n)$, $n=1,\ldots,8$. Let $x_0 = 0 \in \R^4$, and $A = \{z\in\R^4:e^Tz\geq 0\}$, where $e$ is the vector of ones. We calculate $R^{mar}(X)$ for $X\in L^{\infty}_4$ being the payoff of an outperformance option. The vector-valued loss function is taken as $l(x) = (l_1,\ldots, l_4)^T$, with $l_i(x) = e^{x_i}-1$.

We modeled the problem as a CVOP, where there are $4$ objectives, $112$ decision variables, and $109$ constraints. The ordering cone is $K_0 \supsetneq \R^4_+$ with $12$ generating vectors, and we fix $c^1=[1,0,0,0]^T$, $c^2=[0,1,0,0]^T$, $c^3 = [0,0,1,0]^T$, and $c=[1,1,1,1]^T$. We try both algorithms to solve the problem for different values of $\epsilon$. It turns out that the solver fails to solve \eqref{P2} for some of the vertices of the current outer approximation $\mathcal{P}_k$ of the upper image. However, \eqref{P1} can be solved for each $w=w(t)$, for the vertices $t$ of the outer approximations $\mathcal{D}_k$ of the lower image. Thus, Algorithm~\ref{alg_2} provides a finite weak $\epsilon$-solution to \eqref{(P)}, and a finite $\epsilon$-solution to \eqref{(D)}. Also, we only use the variant with `break' as it turns out that it can solve the problem in less time compared to the variant without `break'. Table~\ref{table6} shows some computational data of Algorithm~\ref{alg_2}.

\begin{table}[ht]
\caption{Computational data for Example~\ref{caginex1}}
\centering
\begin{tabular}{ |l|l|l|l|l|l|l|l|l| }
\hline\hline
& & & & & & & & \\ [-2.3ex]
$\epsilon$ & alg. / variant & $\#$ opt. & $\#$ vert. enum. & $\abs{\bar{\mathcal{X}}}$ & $\abs{\bar{\mathcal{T}}}$ & $\abs{\bar{\mathcal{X}}_{alt}}$ & $\abs{\bar{\mathcal{T}}_{alt}}$ & time (s)\\ [0.5ex]
\hline\hline
$0.01$  	& \ref{alg_2} / break  & $152$  & $22$  & $151$   & $148$   & $19$   & $132$   & $3046.5$ \\ 
$0.005$		& \ref{alg_2} / break  & $255$  & $38$  & $254$   & $249$   & $35$   & $219$   & $4941.6$ \\
$0.001$		& \ref{alg_2} / break  & $1433$ & $188$ & $1429$  & $1410$  & $185$  & $1247$  & $26706$ \\
$0.0005$	& \ref{alg_2} / break  & $3109$ & $380$ & $3107$  & $3075$  & $377$  & $2731$  & $58397$ \\ \hline
\end{tabular}
\label{table6}
\end{table}
\end{example}

\bibliographystyle{plain}
\bibliography{reportbib}

\begin{thebibliography}{10}

\bibitem{cagin}
\c{C}. Ararat, A.~H. Hamel, and B.~Rudloff.
\newblock Set-valued shortfall and divergence risk measures.
\newblock {\em submitted}, 2014.

\bibitem{benson}
H.~P. Benson.
\newblock An outer approximation algorithm for generating all efficient extreme
  points in the outcome set of a multiple objective linear programming problem.
\newblock {\em Journal of Global Optimization}, 13:1--24, 1998.

\bibitem{BreFukMar98}
D.~Bremner, K.~Fukuda, and A.~Marzetta.
\newblock Primal-dual methods for vertex and facet enumeration.
\newblock {\em Discrete Computational Geometry}, 20(3):333--357, 1998.

\bibitem{csirmaz}
L.~Csirmaz.
\newblock Using multiobjective optimization to map the entropy region of four
  random variables.
\newblock {\em preprint}, 2013. http://eprints.renyi.hu/66/2/globopt.pdf.

\bibitem{cvx}
Inc. CVX~Research.
\newblock \uppercase{CVX}: Matlab software for disciplined convex programming,
  version 2.0 beta., September 2012.

\bibitem{ehr_dual}
M.~Ehrgott, A.~L\"ohne, and L.~Shao.
\newblock A dual variant of {B}enson's outer approximation algorithm.
\newblock {\em Journal Global Optimization}, 52(4):757--778, 2012.

\bibitem{ehrgott}
M.~Ehrgott, L.~Shao, and A.~Sch\"obel.
\newblock An approximation algorithm for convex multi-objective programming
  problems.
\newblock {\em Journal of Global Optimization}, 50(3):397--416, 2011.

\bibitem{survey_ehrgott}
M.~Ehrgott and M.~M. Wiecek.
\newblock Multiobjective programming.
\newblock In J.~Figueira, S.~Greco, and M.~Ehrgott, editors, {\em Multicriteria
  Decision Analysis: State of the Art Surveys}, pages 667--722. Springer
  Science + Business Media, 2005.

\bibitem{boyd}
M.~C. Grant and S.~P. Boyd.
\newblock Graph implementations for nonsmooth convex programs.
\newblock In V.~Blondel, S.~Boyd, and H.~Kimura, editors, {\em Recent Advances
  in Learning and Control}, volume 371 of {\em Lecture Notes in Control and
  Information Sciences}, pages 95--110. Springer, London, 2008.

\bibitem{lagrange}
A.~H. Hamel and A.~L\"ohne.
\newblock Lagrange duality in set optimization.
\newblock {\em Journal of Optimization Theory and Applications}, 2013, DOI:
  10.1007/s10957-013-0431-4.

\bibitem{lvop}
A.~H. Hamel, A.~L\"ohne, and B.~Rudloff.
\newblock A {B}enson type algorithm for linear vector optimization and
  applications.
\newblock {\em Journal of Global Optimization}, 2013, DOI:
  10.1007/s10898-013-0098-2.

\bibitem{avar}
A.~H. Hamel, B.~Rudloff, and M.~Yankova.
\newblock Set-valued average value at risk and its computation.
\newblock {\em Mathematics and Financial Economics}, 7(2):229--246, 2013.

\bibitem{geometric}
F.~Heyde.
\newblock Geometric duality for convex vector optimization problems.
\newblock {\em Journal of Convex Analysis}, 20(3):813--832, 2013.

\bibitem{geometric_lp}
F.~Heyde and A.~L\"ohne.
\newblock Geometric duality in multiple objective linear programming.
\newblock {\em SIAM Journal of Optimization}, 19(2):836--845, 2008.

\bibitem{freshlook}
F.~Heyde and A.~L\"ohne.
\newblock Solution concepts in vector optimization: a fresh look at an old
  story.
\newblock {\em Optimization}, 60(12):1421--1440, 2011.

\bibitem{jahn}
J.~Jahn.
\newblock {\em Vector Optimization - Theory, Applications, and Extensions}.
\newblock Springer, 2004.

\bibitem{kabanov}
Y.~M. Kabanov.
\newblock Hedging and liquidation under transaction costs in currency markets.
\newblock {\em Finance and Stochastics}, 3:237--248, 1999.

\bibitem{lohne}
A.~L\"ohne.
\newblock {\em Vector Optimization with Infimum and Supremum}.
\newblock Springer, 2011.

\bibitem{superhedging}
A.~L\"ohne and B.~Rudloff.
\newblock An algorithm for calculating the set of superhedging portfolios in
  markets with transaction costs.
\newblock {\em International Journal of Theoretical and Applied Finance},
  Forthcoming, 2013.

\bibitem{luc}
D.~Luc.
\newblock {\em Theory of Vector Optimization}, volume 319 of {\em Lecture Notes
  in Economics and Mathematical Systems}.
\newblock Springer Verlag, 1989.

\bibitem{rockafellar}
R.~T. Rockafellar.
\newblock {\em Convex Analysis}.
\newblock Princeton University Press, 1970.

\bibitem{survey_ruzika}
S.~Ruzika and M.~M. Wiecek.
\newblock Approximation methods in multiobjective programming.
\newblock {\em Journal of Optimization Theory and Applications},
  126(3):473--501, September 2005.

\bibitem{shaoLVOP}
L.~Shao and M.~Ehrgott.
\newblock Approximately solving multiobjective linear programmes in objective
  space and an application in radiotherapy treatment planning.
\newblock {\em Mathematical Methods of Operations Research}, 68(2):257--276,
  2008.

\bibitem{dualalgorithm}
L.~Shao and M.~Ehrgott.
\newblock Approximating the nondominated set of an {MOLP} by approximately
  solving its dual problem.
\newblock {\em Mathematical Methods of Operations Research}, 68(3):469--492,
  2008.

\end{thebibliography}

\end{document}